\documentclass[11pt]{article}
\usepackage[french,english]{babel}
\usepackage{amssymb}
\usepackage{amsmath}
\usepackage{url}
\usepackage[backref,colorlinks=true]{hyperref}
\allowdisplaybreaks

\newenvironment {proof}{{\noindent\bf Proof.}}{\hfill $\Box$ \medskip}

\newtheorem{theorem}{Theorem}[section]
\newtheorem{lemma}[theorem]{Lemma}

\newtheorem{proposition}[theorem]{Proposition}
\newtheorem{remark}[theorem]{Remark}
\newtheorem{definition}[theorem]{Definition}

\newtheorem{corollary}[theorem]{Corollary}

\newtheorem{assumption}[theorem]{Assumption}

\renewcommand {\theequation}{\arabic{section}.\arabic{equation}}

\def \tilde{\widetilde}
\def \bar{\overline}

\newcommand{\Co}{C}

\newcommand{\supp}{\mbox{supp}}

\def\D{\mathbb{D}}

\def\N{\mathbb{N}}
\def\P{\mathbb{P}}
\def\R{\mathbb{R}}
\def\E{\mathbb{E}}
\def\X{\mathbb{X}}
\def\M{\mathcal{M}}

\def\supp{\textrm{supp}}
\def\Fit{\textrm{Fit}}
\def\ind{{\mathchoice {\rm 1\mskip-4mu l} {\rm 1\mskip-4mu l}
{\rm 1\mskip-4.5mu l} {\rm 1\mskip-5mu l}}}

\title{A new proof for the convergence of an individual based model to the Trait substitution sequence}

\author{Ankit Gupta\thanks{CMAP, Ecole Polytechnique, UMR CNRS 7641, Route de Saclay, 91128 Palaiseau C\'{e}dex, France} ~ and~ J.A.J. Metz\thanks{Mathematical Institute \& Institute of Biology \& NCB Naturalis, Leiden, Niels Bohrweg 1, 2333 CA, Leiden, Netherlands ; EEP, IIASA, Laxenburg, Austria} ~ and ~ Viet Chi Tran\thanks{Equipe Probabilit\'{e} Statistique, Laboratoire Paul Painlev\'{e}, UMR CNRS 8524, UFR de Math\'{e}matiques, Universit\'{e} des Sciences et Technologies Lille 1, Cit\'{e} Scientifique, 59655 Villeneuve d'Ascq C\'{e}dex, France ; CMAP, Ecole Polytechnique}}

\date{\today}

\begin{document}

\maketitle

\begin{abstract}
We consider a continuous time stochastic individual based model for a population structured only by an inherited vector trait and with logistic interactions. We consider its limit in a context from adaptive dynamics: the population is large, the mutations are rare and we view the process in the timescale of mutations. Using averaging techniques due to Kurtz (1992), we give a new proof of the convergence of the individual based model to the trait substitution sequence of Metz et al. (1992) first worked out by Dieckman and Law (1996) and rigorously proved by Champagnat (2006): rigging the model such that ``invasion implies substitution'', we obtain in the limit a process that jumps from one population equilibrium to another when mutations occur and invade the population. \end{abstract}

\noindent Keywords: birth and death process; structured population; adaptive dynamics; individual based model; averaging technique; trait substitution sequence\\
\noindent Mathematical Subject Classification (2000): 92D15; 60J80; 60K35; 60F99
\medskip

\section{Introduction: The logistic birth and death model}

We consider a stochastic individual based model (IBM) with trait structure and evolving as a result of births and deaths, that has been introduced by Dieckmann and Law \cite{DieckmannLaw1996} and Metz et al. \cite{metzgeritzmeszenajacobsheerwaarden} and in rigorous detail by Champagnat \cite{champagnat3}. We study its limit in an evolutionary time scale when the population is large and the mutations are rare.


Champagnat \cite{champagnat3} established the first rigorous proof of the convergence of a sequence of such IBMs to the trait substitution sequence process (TSS) introduced by Metz et al.\cite{metznisbetgeritz} (with Metz et al. \cite{metzgeritzmeszenajacobsheerwaarden} as follow up) which, following Dieckmann and Law \cite{DieckmannLaw1996}, can be explained as follows.
In the limit the time scales of ecology and evolution are separated. Mutations are rare and before a mutant arises, the resident population stabilizes around an equilibrium. Under the ``invasion implies substitution'' Ansatz, there cannot be long term coexistence of two different traits. Evolution thus proceeds as a succession of monomorphic population equilibria. The fine structure of the transitions disappear on the time scale that is considered and, when a mutation occurs, invades, and fixates in the population by completely replacing the resident trait, the TSS jumps from one state to another. Champagnat's proof is based on some fine estimates, including some fine large deviations results, to combine several approximations of the microscopic process. We separate the different time scales that are involved using averaging techniques due to Kurtz \cite{kurtz},  and thus propose a new simplified proof of Champagnat's results that skips many technicalities linked with fine approximations of the IBM. The aim of this paper is to exemplify the use of such averaging techniques in adaptive dynamics, which we hope will pave the way for generalizations of the TSS.\\

We consider a structured population where each individual is characterized by a trait $x\in \X$, a compact subset of $\R^d$.
We are interested in large populations. We assume that the population's initial size to be proportional to a parameter $K\in \N^*=\{1,2,\dots \}$, to be interpreted as the area to which the population is confined, that we will let go to infinity while keeping the density constant by counting individuals with weight $1/K$. The population is assumed to be well mixed and its density is assumed to be limited by a fixed availability of resources per unit of area. The population at time $t$ can be described by the following point measure on $\X$
\begin{align}
\label{def:nut}
X^{K}(t) = \frac{1}{K} \sum_{i=1}^{N^{K}(t)} \delta_{x^i_t},
\end{align}
where $N^{K}(t)$ is the total number of individuals in the population at time $t$ and where $x^i_t\in \X$ denotes the trait of individual $i$ living at time $t$, the latter being numbered in lexicographical order.\\

The population evolves by births and deaths. An individual with trait $x\in \X$ gives birth to new individuals at rate $b(x)$, where $b(x)$ is a continuous positive function on $\X$. With probability $u_K p(x)\in [0,1]$, the daughter is a mutant with trait $y$, where $y$ is drawn from the mutation kernel $m(x,dy)$ supported on $\X$. Here $u_K \in [0,1]$ is a parameter depending on $K$ that scales the probability of mutation.
With probability $1-u_K p(x)\in [0,1]$, the daughter is a clone of her mother, with the same value of the trait $x$. In a population described by $X\in \mathcal{M}_F(\X)$, the individual with trait $x$ dies at rate $d(x)+\int_{\X}\alpha(x,y)X(dy)$, where the natural death rate $d(x)$ and the competition kernel $\alpha(x,y)$
are positive continuous functions.

\begin{assumption}\label{hyp:taux}We assume that the functions $b$, $d$ and $\alpha$ satisfy the following hypotheses:
\begin{itemize}
\item[(A)] For all $x\in \X$, $b(x)-d(x)>0$ and $p(x)>0$.
\item[(B)] ``Invasion implies substitution'': For all $x$ and $y$ in $\X$, we either have:
\begin{align}
& (b(y)-d(y))\alpha(x,x)-(b(x)-d(x))\alpha(y,x)<0 \\
\mbox{ or }\qquad & \left\{\begin{array}{l}
(b(y)-d(y))\alpha(x,x)-(b(x)-d(x))\alpha(y,x)>0 \\
(b(x)-d(x))\alpha(y,y)-(b(y)-d(y))\alpha(x,y)<0.
\end{array}\right.\end{align}
\item[(C)] There exist $\underline{\alpha}$ and $\bar{\alpha}>0$ such that for every $x,y\in \X$:
\begin{equation}
0<\underline{\alpha}\leq \alpha(x,y)\leq \bar{\alpha}.
\end{equation}
\end{itemize}
\end{assumption}

Part (A) of Assumption \ref{hyp:taux} says that in the absence of competition, the population has a positive natural growth rate. Also the probability of a birth resulting in a mutation in positive. Part (B) corresponds to a condition known in adaptive dynamics as ``invasion implies substitution''. It can be obtained from the analysis of the equiliria of the Lotka-Volterra system that results from the ordinary large number limit of the logistic competition process without mutation.The condition has as consequences that  when a mutant population manages to reach a sufficiently large size, it cannot coexist with the resident population: one of the two types has to become extinct. Hence, the population should be monomorphic away from the mutation events.\\

In this paper, we use the methods of Kurtz \cite{kurtzaveraging}, based on martingale problems, for separating time scales in measure-valued processes. We show that on the one hand the populations stabilize around their equilibria on the fast ecological scale, while on the other hand, rare mutations at the evolutionary time scale may induce switches from one trait to another.
This provides a new proof to the result of Champagnat \cite{champagnat3}, which we hope may pave the way to generalizations of the TSS.
Our proof differs from the one in  \cite{champagnat3} in that we do not require comparisons with partial differential equations and large deviation results to exhibit
the stabilization of the population around the equilibria determined by the resident trait.

In Section \ref{section:IBM}, we describe the IBM introduced in \cite{champagnat3}. The model accounts only for a trait-structure and otherwise has very simple dynamics. Generalization to more general (possibly functional, as first introduced in \cite{DieckmannHeinoParvinen,ParvinenDieckmannHeino}) trait spaces are considered in \cite{guptametztran} where the ``invasion implies substitution'' assumption is also relaxed (see also \cite{metzgeritzmeszenajacobsheerwaarden,geritzkisdimescenametz,champagnatmeleard2011}). There exist several other possible generations of the TSS, including e.g.  physiological structure \cite{durinxmetzmeszena,meleardtran} or diploidy \cite{colletmeleardmetz}. We consider the process that counts the new traits appearing due to mutations, and the occupation measure $\Gamma^K$ of the process $X^K$ under a changed time scale. The tightness of the couple of processes is studied in Section \ref{section:tightness}. The limiting values are shown to satisfy an equation that is considered in Section \ref{section:characterization}. This equation says that when a favorable mutant appears, then the distribution describing the population jumps to the equilibrium characterized by the new trait with a probability depending on the fitness of the mutant trait compared to the resident one. From the consideration of monomorphic and dimorphic populations and using couplings with branching processes, we prove the convergence in distribution of $\{\Gamma^K\}$ to the occupation measure $\Gamma$ of a pure jump Markov process that is called TSS.

\medskip

\noindent \textbf{Notation:}\\
Let $E$ be a Polish space and let $\mathcal{B}(E)$ be its Borel sigma field. We denote by $\mathcal{M}_F(E)$ (resp. $\mathcal{M}_P(E)$) the set of nonnegative finite (resp. point) measures on $E$, endowed with the topology of weak convergence. If $E$ is compact this topology coincides with the topology of vague convergence (see e.g. \cite{kallenberg}). If $E$ is compact, then for any $M>0$, the set $\{\mu \in \mathcal{M}_F(E) : \mu(E) \leq M\}$ is compact in this topology. For a measure $\mu$, we denote its support by $\supp(\mu)$. If $f$ is a bounded measurable function on $E$ and $\mu \in \mathcal{M}_F(E)$, we use the notation: $\langle \mu,f\rangle=\int_{E} f(x)\mu(dx)$. With an abuse of notation, $\langle \mu,x\rangle=\int_E x \mu(dx)$.
Convergence in distribution of a sequence of random variables (or processes) is denoted by `$\Rightarrow$'. The minimum of any two numbers $a,b \in \R$ is given by $a \wedge b$ and for any $a\in \R$, its positive part is denoted by $[a]^{+}$. For any two $\N^*$-valued sequences $\{a_K : K \in \N^*\}$ and $\{b_K : K \in \N^*\}$ we say that $a_K \ll b_K$ if $a_K/b_K \to 0$ as $K \to \infty$.

Define a class of test functions on $\mathcal{M_F}(\X)$ by
\begin{align*}
\mathbb{F}^{2}_b = \{ F_f : F_f(\mu) = F (\langle \mu, f \rangle ), f \in \Co_b(\X,\R) \textrm{ and } F \in \Co^2_b(\R,\R) \textrm{ with compact support}\}.
\end{align*}
Here $\Co_b(\X,\R)$ is the set of all continuous and bounded real functions on $\X$ and $\Co^2_b(\R,\R)$ is the set of bounded, twice continuously differentiable real-valued functions on $\R$ with bounded first and second derivative. This class $\mathbb{F}^{2}_b$ is separable and it is known (see for example \cite{dawson}) that it characterizes the convergence in distribution on $\M_F(\X)$.

The value at time $t$ of a process $X$ is denoted $X(t)$ or sometimes $X_t$ for notational convenience.

\section{IBM in the evolutionary time-scale}\label{section:IBM}

The process $X^K$ is characterized by its generator $L^K$, defined as follows. To simplify notation, we define the transition kernel:
\begin{equation}
M^K(x,dy)=u_K p(x)\ m(x,dy)+(1-u_K p(x))\ \delta_{x}(dy).
\end{equation}
For any $F_f \in \mathbb{F}^{2}_b$ let
\begin{align}
\label{genlk}
L^K F_{f}(X) & =K \int_{E} b(x) \left[  \int_{\X} \left( F_f\left( X+ \frac{1}{K} \delta_{y} \right) - F_f(X) \right)M^K(x,dy)\right] X(dx) \notag \\
&+K \int_{E} \Big(d(x)+\langle X,\alpha(x,.)\rangle\Big)  \left( F_f\left( X- \frac{1}{K} \delta_{x} \right) - F_f(X) \right)  X(dx).
\end{align}
Let $K\in \N^*$ be fixed. The martingale problem for $L^K$ has a unique solution for any initial condition $X^K(0) \in  \mathcal{M}_{F}(\X)$. It is possible to construct the solution of the martingale problem by considering a stochastic differential equation (SDE) driven by Poisson point processes and which corresponds to the IBM used for simulations (see \cite{champagnatferrieremeleard2,champagnatferrieremeleard}).
We need the following estimate to proceed, which is shown in \cite[Lemma 1]{champagnat3}
\begin{lemma}
\label{lemma:boundednessofmass}
Suppose that $\sup_{K \in \N^* } \E\big( \langle  X^K(0),1 \rangle^{2}\big) < \infty$, then
 \[ \sup_{K\geq 1,\ t\geq 0}\E\Big(\langle X^K(t),1\rangle^2\Big)<+\infty.\]
\end{lemma}
In the sequel, we hence make the following assumptions about the initial condition.
\begin{assumption}\label{assumption:monomorphicIC}
Suppose that the sequence of $\mathcal{M}_F(\X)$-valued random variables $\{X^K(0) : K \in \N^*\}$ satisfy the following conditions.
\begin{itemize}
\item[(A)] There exists $x_0 \in \X$ such that $\textrm{supp}(X^K(0)) =\{x_0\}$ for all $K \in \N^*$.
\item[(B)] $\sup_{K\in \N^*}\E\left(\langle X^K(0),1\rangle^2\right) < \infty$.
\item[(C)] $X^K(0) \Rightarrow X(0)$ as $K \to \infty$ and $\langle X(0), 1 \rangle > 0$ a.s.
\end{itemize}
\end{assumption}

From \eqref{genlk}, we can see that the dynamics has two time scales. The slower time scale is of order $Ku_K$ and it corresponds to the occurence of new mutants while the faster time scale is of order $1$ and it corresponds to the birth-death dynamics. We consider rare mutations and will therefore be lead to suppose that for any $c>0$
\begin{assumption}
\label{timescaleseparation}
\[ \log K \ll \frac{1}{K u_K } \ll e^{cK}.\]
\end{assumption}

Consider the process
\begin{align}
\label{defZK}
Z^K(t) = X^K \left( \frac{t}{Ku_K} \right), \textrm{ } t \geq 0.
\end{align}
In what follows, we denote by $\{ \mathcal{F}^K_t : t\geq 0\}$ the canonical filtration associated with $Z^K$. Due to the change in time, the generator $\mathbb{L}^K$ of $Z^K$ is the generator $L^K$ of $X^K$ multiplied by $(1/Ku_K)$. Hence for any $F_f \in \mathbb{F}^{2}_b$
\begin{align}
\mathbb{L}^K  F_f(Z)= &  \int_{\X} p(x) b(x) \left[ \int_{\X} \left( F_f\big(Z+\frac{1}{K}\delta_{y}\big)-F_f(Z) \right) m(x,dy)\right]Z(dx)
\nonumber\\
& + \frac{1}{K u_K} \left[  \int_{\X} b(x) (1-u_K p(x)) K \left( F_f\left( Z+ \frac{1}{K} \delta_{x} \right) - F_f(Z) \right) Z(dx) \right.  \nonumber\\
 & \left. +  \int_{\X} \left(d(x)+\langle Z,\alpha(x,.)\rangle \right)  K \left( F_f\left( Z- \frac{1}{K} \delta_{x} \right) - F_f(Z) \right) Z(dx)\right].\label{generateur:Lgras}
\end{align}
In the process $Z^K$ we have compressed time so that new mutants occur at rate of order $1$. When we work at this time scale we can expect that between subsequent mutations, the fast birth-death dynamics will average (see e.g. \cite{kurtzaveraging}). Our aim is to exploit this separation in the time scales of ecology (which is related to the births and deaths of individuals) and of evolution (which is linked to mutations).

To study the averaging phenomenon, for the fast birth-death dynamics, we use the martingale techniques developed by Kurtz \cite{kurtzaveraging}. We introduce the occupation measure $\Gamma^K$ defined for any $t\geq 0$ and for any set $A \in \mathcal{B}\left( \mathcal{M}_F(\X) \right)$ by
\begin{align}
\Gamma^K([0,t] \times A) = \int_{0}^{t} \ind_{A}(Z^K(s))ds.
\end{align}
Kurtz's techniques have been used in the context of measure-valued processes in \cite{meleardtransuperage,meleardmetztran} for different population dynamic problems, but an additional difficulty arises here due to the presence of non-linearities also at the fast time scale.

We introduce a $\mathcal{M}_P(\X)$-valued process $\{\chi^K(t) : t \geq 0\}$ which keeps track of the traits that have appeared in the population. That is, for each $t \geq 0$, $\chi^K(t)$ is a counting measure on $\X$ that weights the traits that have appeared in the population until time $t$. The process $\chi^K$ is a pure-jump Markov process that satisfies the following martingale problem. For any $F_f \in \mathbb{F}^{2}_b$
\begin{align}
M^{\chi,K}_t:= & F_f(\chi^K(t))-F_f(\chi^K(0))\\
- & \int_0^t  \int_{\X} p(x) b(x) \int_{\X}\left(F_f(\chi^K(s)+\delta_{y})-F_f(\chi^K(s))\right) m(x,dy)\, Z^K(s,dx)ds\nonumber\\
= &  F_f(\chi^K(t))-F_f(\chi^K(0)) \notag \\
- & \int_0^t  \int_{\mathcal{M}_F(\X)}\left[ \int_{\X} p(x) b(x) \int_{\X}\left( F_f(\chi^K(s)+\delta_{y})-F_f(\chi^K(s))\right) m(x,dy) \mu(dx)\right]\, \Gamma^K(ds \times d\mu),\label{martingale:chi2}
\end{align}
is a square integrable $\{\mathcal{F}^K_t\}$-martingale.\\

The main result of the paper proves the convergence of $\{(\chi^K,\Gamma^K)\}$ to a limit $(\chi,\Gamma)$, where the slow component $\chi$ is a jump Markov
process and the fast component stabilizes in an equilibrium that depends on the value of the slow component.

\begin{theorem}\label{theorem:convergenceTSS}
Suppose that Assumptions \ref{hyp:taux} and \ref{assumption:monomorphicIC} hold.
 \begin{itemize}
 \item[(A)] There exists a process $\chi$ with paths in $\D \left( \mathcal{M}_P(\X),[0,\infty)  \right)$ and a random measure $\Gamma \in \mathcal{M}_F \left([0,\infty)\times \M_F(\X) \right)$ such that $\left( \chi^K, \Gamma^K \right) \Rightarrow \left(\chi,\Gamma \right)$ as $K \to \infty$, where $(\chi,\Gamma)$ are characterized as follows. For all functions $F_f \in \mathbb{F}^{2}_b$,
\begin{multline}
\label{limitingmartingale1}
F_f(\chi(t)) - F_f(\chi(0))\\
-\int_{0}^{t} \int_{\mathcal{M}_F(\X)} \left[\int_\X b(x)p(x)  \int_{\X} \Big(F_f(\chi(s)+\delta_{y})-F_f(\chi(s))\Big) m(x,dy) \mu(dx)\right]\Gamma (ds \times d \mu)
\end{multline}is a square integrable martingale for the filtration
\begin{align}
\label{deffiltration}
\mathcal{F}_t = \sigma \left\{\chi(s), \Gamma\left([0,s] \times A \right) : s \in [0,t], A \in \mathcal{B}\left(\mathcal{M}_F(\X)\right) \right\}
\end{align}and for any $t \geq 0$
\begin{align}
\label{limitingmartingale2}
 \int_{0}^{t}\int_{\mathcal{M}_F(\X)} \mathbb{B} F_f (\mu)\Gamma(ds \times d\mu) = 0 \textrm{ for all }F_f \in \mathbb{F}^{2}_b \textrm{ a.s.},
\end{align}
where the nonlinear operator $\mathbb{B}$ is defined by
\begin{align}
\mathbb{B}F_f(\mu) =  F' \left( \left\langle \mu ,f\right\rangle \right)\int_{ \X } \left( b(x)   - \left( d(x)+\langle \mu, \alpha(x,.)\rangle \right)  \right) f(x) \mu(dx).\label{populationdynamicsoperator}
\end{align}
 \item[(B)] Moreover for any $t>0$ and $A \in \mathcal{B}\left( \mathcal{M}_F(\X)\right)$ we have
\begin{align}
\label{occupationmeasurerelation}
\Gamma([0,t] \times A) = \int_{0}^{t} \ind_{A}\left( \widehat{n}_{\chi'(s)} \delta_{\chi'(s)}\right)ds,
\end{align}
where $\{\chi'(t) : t \geq 0\}$ is a $\X$-valued Markov jump process  with $\chi'(0) = x_0$ and generator given by
\begin{align}
\label{generatorc}
\mathbb{C} f(x) =  b(x)p(x) \widehat{n}_{x} \int_{\X}   \frac{[\mathrm{Fit}(y,x)]^{+}}{b(y)} \left( f(y) - f(x) \right) m(x,dy),
\end{align}
for any $f \in \Co_b(\X,\R)$. Here the population equilibrium $\widehat{n}_x$ and the fitness function $\Fit(y,x)$ are given by:
\begin{equation}
\widehat{n}_x=(b(x)-d(x))/\alpha(x,x),\qquad \mbox{and}\qquad \mathrm{Fit}(y,x) = b(y) - d(y) - \alpha(y,x)\widehat{n}_x .\label{nchapeaux}
\end{equation}
\end{itemize}
\end{theorem}

This fitness function provides information on the stability of the Lotka-Volterra system (see \eqref{ode} below) and describes the growth rate of a negligible mutant population with trait $y$ in an environment characterized by $\widehat{n}_x$.

\section{Tightness of $\{ \left( \chi^K,\Gamma^K \right)\}$ }\label{section:tightness}

To study the limit when $K\rightarrow +\infty$, we proceed by a tightness-uniqueness argument. First, we show the tightness of the distributions of $\{ \left( \chi^K,\Gamma^K \right) :  K \in \N^*\}$ and derive certain properties of the limiting distribution. The limiting values of $\{ \Gamma^K\}$ satisfy an equation that characterizes the state of the population between two mutations, thanks to the ``invasion implies substitution'' assumption.
\begin{theorem}\label{maintheorem}
Suppose that Assumption \ref{hyp:taux} is satisfied and $\sup_{K\geq 1}\E(\langle X^K(0),1\rangle^2)<\infty$. Then:
\begin{itemize}
\item[(A)] The distributions of $\{ (\chi^K,\Gamma^K) : K \in \N^*\}$ are tight in the space $$\mathcal{P}\left( \D \left(\R_+, \mathcal{M}_P(\X) \right) \times  \mathcal{M}_F \left(\R_+\times \M_F(\X) \right)\right).$$
\item[(B)] Suppose that $(\chi^K,\Gamma^K) \Rightarrow (\chi,\Gamma)$, along some subsequence, as $K \to \infty$. Then $\chi$ is characterized by the  martingale problem given by \eqref{limitingmartingale1} and $\Gamma$ satisfies the equation \eqref{limitingmartingale2}.
\end{itemize}
\end{theorem}

In fact, the Assumption \ref{timescaleseparation} can be relaxed and the proof of Theorem \ref{maintheorem} only requires that $\lim_{K\rightarrow +\infty} Ku_K =0$.\\

\begin{proof}[Proof of Th. \ref{maintheorem}]
To prove the tightness of $\{\chi^K : K\geq 1\}$, we use a criterion from \cite{ethierkurtz}. Let $n^K(t) = \langle \chi^K (t),1 \rangle$ for $t \geq 0$.
This process counts the number of mutations in the population. For any $T>0$ we have by Lemma \ref{lemma:boundednessofmass} that
\begin{align}
\label{tighness:expectednumberofmutations}
 \sup_{K \geq 1} \E ( n^K(T)) \leq & \|b\|_\infty T  \sup_{K \geq 1, t\leq T}\E\left( \langle  Z^K(t),1\rangle\right)\leq \|b\|_\infty  T  \sup_{K\geq 1, t \geq 0} \E\left( \langle X^K(t),1\rangle\right)< \infty.
\end{align}
From this estimate and the martingale problem \eqref{martingale:chi2}, it can be checked by using Aldous and Rebolledo criteria that for every test function $f\in \Co_b(\X,\R)$, the laws of $\langle \chi^K, f\rangle$ are tight in $\D([0,\infty),\R)$ and that the compact containment condition is satisfied.

Let us now prove the tightness of $\{\Gamma^K : K \in \N^*\}$. Let $\epsilon>0$ be fixed. Using Lemma \ref{lemma:boundednessofmass}, there exists a $N_{\epsilon}>0$ such that
\begin{align}
\label{etape6}
\sup_{K \geq 1, t \geq 0} \P(\langle Z^K(t),1\rangle > N_{\epsilon})<\epsilon.
\end{align}
Since $\X$ is compact, the set $\mathcal{K}_{\epsilon}=\{\mu\in \M_F(\X),\ \langle \mu,1\rangle\leq N_{\epsilon}\}$ is compact. We deduce that for any $T>0$
\begin{equation}
 \inf_{K \geq 1} \E \left( \Gamma^K\big([0,T]  \times \mathcal{K}_{\epsilon}  \big)\right) \geq (1-\epsilon)T.\label{controlekurtz}
\end{equation}
Indeed
\begin{align*}
\Gamma^K\big([0,T]  \times \mathcal{K}_{\epsilon}  \big)= & \Gamma^K\big([0,T]  \times \mathcal{M}_F(\X)\big)-\Gamma^K\big([0,T]  \times \mathcal{K}_{\epsilon}^c  \big)=T-\int_0^T \ind_{\mathcal{K}_{\epsilon}^c}(Z^K(t))dt
\end{align*}and the result follows from the Fubini theorem and from \eqref{etape6}. From Lemma 1.3 of \cite{kurtzaveraging}, $\{\Gamma^K : K \in \N^*\}$ is a tight family of random measures.
The joint tightness of $\{\left( \chi^K,\Gamma^K \right) : K\in \N^* \}$ is immediate from the tightness of $\{\chi^K : K \in \N^* \}$ and
$\{\Gamma^K : K \in \N^*\}$. This proves part (A).

We now prove part (B). Our proof is adapted from the proof of Theorem 2.1 in \cite{kurtzaveraging}. From part (A) we know that the distributions
$\{(\chi^K,\Gamma^K) : K \in \N^*\}$ are tight. Therefore there exists a subsequence $\{\eta_K\}$ along which $(\chi^K,\Gamma^K)$ converges in distribution to a limit $(\chi,\Gamma)$. We can take the limit in \eqref{martingale:chi2} along this subsequence. Except for a denumberable subset of times, $M_t^{\chi,K}$ converges in distribution to the martingale given by \eqref{limitingmartingale1}.

Let us now show that the limiting value $\Gamma$ satisfies \eqref{limitingmartingale2}. From \eqref{generateur:Lgras} for any $F_f \in \mathbb{F}^{2}_b$, we get that
\begin{align}
m^{F,f,K}_t & =  F_f\left(Z^K(t) \right)-F_f \left(Z^K(0) \right)- \int_0^t \mathbb{L}^K F_f(Z^K(s)) ds  \notag \\
& = F_f\left(Z^K(t) \right) - F_f \left(Z^K(0) \right)- \left(\frac{1}{K u_K} \right)\int_{0}^{t} \int_{\mathcal{M}_F(\X)} \mathbb{B} F_f(\mu) \Gamma^K(ds \times d\mu)- \frac{\delta^{F,f,K}(t)}{K u_K}  \label{martingale:second2}
\end{align}
is a martingale. Here the operator $\mathbb{B}$ is defined by \eqref{populationdynamicsoperator} and
\begin{align}
\label{defndeltaK}
\delta^{F,f,K}(t)=  \int_{0}^{t} \left(K u_K \mathbb{L}^K F_f(Z^K(s)) - \mathbb{B} F_f(Z^K(s))\right)ds.
\end{align}
For any $\mu \in \mathcal{M}_{F}(\X)$ we have
\begin{multline*}
 K u_K \mathbb{L}^K F_f(\mu) - \mathbb{B} F_f(\mu)\\
  \begin{aligned}
& = K u_K \int_{\X} p(x) b(x)\left[ \int_{\X } \left( F_f\left(\mu+\frac{1}{K}\delta_{y}\right)-F_f(\mu) \right) m(x,dy)\right]\mu(dx)  \\
& + K \int_{\X}  b(x) \left( F_f\left(\mu+ \frac{1}{K} \delta_{x} \right) - F_f(\mu) - \frac{1}{K} F'\left( \langle \mu,f  \rangle\right)f(x) \right) \mu(dx)\\
& +  K \int_{\X} \left( d(x) + \langle \mu, \alpha(x,.)\rangle \right)  \left( F_f\left( \mu- \frac{1}{K} \delta_{x} \right) - F_f(\mu) + \frac{1}{K} F'\left( \langle  \mu , f \rangle\right)f(x) \right)  \mu(dx) \\
& - K u_K   \int_{\X}  p(x)b(x) \left( F_f\left( \mu+ \frac{1}{K} \delta_{x} \right) - F_f(\mu) \right)\mu(dx).
\end{aligned}
\end{multline*}
For any $x\in \X$ and $\mu\in \mathcal{M}_F(\X)$, we have by Taylor expansion that for some $\alpha_1, \alpha_2 \in (0,1)$:
\begin{align*}
 F_f\left( \mu \pm \frac{1}{K} \delta_{x} \right) - F_f(\mu) 
& = \pm F' \left(\langle \mu, f\rangle + \alpha_1 \frac{f(x)}{K} \right) \frac{f(x)}{K} \\
& = \pm F' (\langle \mu,f\rangle) \frac{f(x)}{K} + \frac{f(x)^2}{2K^2} F''\left(\langle \mu,f\rangle+\alpha_2 \frac{f(x)}{K}\right).
\end{align*}
Therefore we get
\begin{align*}
\left| F_f\left( \mu \pm \frac{1}{K} \delta_{x} \right) - F_f(\mu) \right| \leq   \frac{\|F'\|_\infty \|f\|_{\infty}^2}{K}.
\end{align*}
and
\begin{align*}
\left|  K\left( F_f\left( \mu \pm \frac{1}{K} \delta_{x} \right) - F_f(\mu)  \mp \frac{1}{K} F' (\langle \mu,f\rangle)f(x) \right) \right| \leq   \frac{\|F''\|_\infty \|f\|_{\infty}^2}{2K}.
\end{align*}
Using these estimates and Assumption \ref{hyp:taux},
\begin{align*}
\left| K u_K \mathbb{L}^K F_f(\mu) - \mathbb{B} F_f(\mu) \right|& \leq 2 u_K \left\|b\right\|_{\infty} \|F'\|_\infty \|f\|_{\infty}^2  \langle  \mu,1 \rangle  \\
& + \frac{\|F''\|_\infty \|f\|_{\infty}^2}{2K} \left( \left(\left\|b\right\|_{\infty} +\left\|d\right\|_{\infty}\right) \langle \mu,1 \rangle + \bar{\alpha}  \langle \mu,1 \rangle^2   \right).
\end{align*}
Pick any $T>0$. This estimate along with Lemma \ref{lemma:boundednessofmass} implies that as $K\rightarrow +\infty$, $\delta^{F,f,K}(t)$ (given by \eqref{defndeltaK}) converges to $0$ in $L^1(d\P)$, uniformly in $t\in [0,T]$. Multiplying \eqref{martingale:second2} by $K u_K$, we get that along the subsequence $\eta_K$, the sequence of martingales $\{Ku_K m^{F,f,K} : K \in \N^*\}$ converge in $L^1(d\P)$, uniformly in $t\in [0,T]$ to $\int_{0}^{t} \int_{\mathcal{M}_F(\X)} \mathbb{B} F_f (\mu)\Gamma(ds \times d\mu). $
The limit is itself a martingale. Since it is continuous and has paths of bounded variation, it must be $0$ at all times a.s. Hence for any $F_f \in \mathbb{F}^{2}_b$,
\[\int_{0}^{t} \int_{\mathcal{M}_F(\X)} \mathbb{B} F_f (\mu)\Gamma(ds \times d\mu) = 0 \textrm{ a.s.}\]
Since $\mathbb{F}^{2}_b$ is separable (\ref{limitingmartingale2}) also holds.
\end{proof}

\section{Characterization of the limiting values}\label{section:characterization}

\subsection{Dynamics without mutation}

As in \cite{champagnat3,champagnatferrieremeleard2}, to understand of the information provided by \eqref{limitingmartingale2}, we are led to consider the dynamics of monomorphic and dimorphic populations. Our purpose in this section is to show that \eqref{populationdynamicsoperator} and Assumption \ref{hyp:taux} (B) characterize the state of the population between two rare mutant invasions. Because of  Assumption \ref{hyp:taux} (B), we will see that two different traits cannot coexist in the long term and thus, we will see that it suffices to work with monomorphic or dimorphic initial populations (i.e. the support of $Z^K_0$ is one or two singletons).\\
In Subsection \ref{subsect1}, we consider monomorphic or dimorphic populations and show convergence of the occupation measures when the final composition in trait of the population is known. For instance, if the remaining trait is $x_0$, then the occupation measure of $Z^K(dx,dt)$ converges to $\widehat{n}_{x_0} \delta_{x_0}(dx)dt$. In Subsection \ref{subsect2}, we use couplings with birth and death processes to show that the distribution of the final trait composition of the population can be computed from the fitness of the mutant and the resident.

\subsubsection{Convergence of the occupation measure $\Gamma^K$ in absence of mutation}\label{subsect1}

First, we show that the ``invasion implies substitution'' Assumption \ref{hyp:taux} (B) provides information on the behavior of a dimorphic process when we know which trait is fixed.

\begin{definition}\label{def:YKL0}
Let $L^K_0$ be the operator $L^K$ (given by \eqref{genlk}) with $p(x) = 0$ for all $x \in \X$. We will denote by $Y^K$ a process with generator $L^K_0$ and with a initial condition that varies according to the case that is studied. $Y^K$ has the same birth-death dynamics as a process with generator $L^K$, but there are no mutations.
\end{definition}

In this section we investigate how a process with generator $L^K_0$ behaves at time scales of order $1/K u_K$, when the population is monomorphic or dimorphic. We start by proving a simple proposition.

\begin{proposition}
\label{prop:identification}
For any $x,y \in \X$ suppose that $\pi \in \mathcal{P}\left( \mathcal{M}_F(\X) \right)$ is such that
\begin{align}
\label{support}
\pi \left( \left\{  \mu \in \mathcal{M}_{F}(\X) : \{x\} \subset \textrm{supp}(\mu) \subset \{x,y\} \right\} \right) = 1
\end{align}
and \begin{equation}
\int_{\mathcal{M}_F(\X)} \mathbb{B}  F_f (\mu)\pi(d\mu) = 0\label{equation:B_pi}
 \end{equation}for all $F_f \in \mathbb{F}^{2}_b$. Then for any $A \in \mathcal{B}\left( \mathcal{M}_F(\X) \right)$ we have
$\pi(A) = \ind_{A}\left( \widehat{n}_x \delta_{x} \right)$ where $\widehat{n}_x$ has been defined in \eqref{nchapeaux}.
\end{proposition}

\begin{proof}
Since $\pi$ satisfies \eqref{support}, any $\mu$ picked from the distribution $\pi$ has the form $\mu = n_x \delta_{x} + n_y \delta_{y}$
with $n_x >0$. Let $\Phi$ be the map from $\mathcal{M}_{F}(\X)$ to $\R_{+} \times \R_{+}$ defined by
\[ \Phi(\mu) = \left( \langle \mu, \ind_{\{x\}} \rangle , \langle \mu, \ind_{\{y\}}  \rangle  \right).\]
Let $\pi^{*}=\pi \circ \Phi^{-1} \in \mathcal{P} \left( \R_{+} \times \R_{+} \right)$ be the image distribution of $\pi$ by $\Phi^{-1}$. From \eqref{equation:B_pi}, replacing $\mathbb{B} F_f$ by its definition we obtain
\begin{align*}
0 & = \int_{\mathcal{M}_F(\X)}F'(\langle \mu ,f\rangle) \left(\langle \mu, (b-d)f\rangle -\int_{E} f(x) \langle \mu, \alpha(x,.)\rangle \mu(dx)\right) \pi(d\mu) \\
& = \int_{\R_{+} \times \R_{+}} F'\left( f(x)n_x+f(y)n_y\right) \left[ \left(  b(x)-d(x) - \alpha(x,x)n_x - \alpha(x,y)n_y  \right) n_x f(x) \right. \\
& \left.+ \left(  b(y)-d(y) - \alpha(y,x)n_x - \alpha(y,y)n_y  \right) n_y f(y) \right] \pi^{*}(dn_x,dn_y).
\end{align*}
This equation holds for all $F_f \in \mathbb{F}^{2}_b$ only if the support of $\pi^{*}$ consists of $(n_x,n_y)$ with $n_x>0$ that satisfy
$\left(  b(x)-d(x) - \alpha(x,x)n_x - \alpha(x,y)n_y  \right) n_x = 0$ and $\left(  b(y)-d(y) - \alpha(y,x)n_x - \alpha(y,y)n_y  \right) n_y = 0$.
The only possible solutions are $(\widehat{n}_x,0)$ and
\begin{align}
&(\widetilde{n}_x,\widetilde{n}_y)\\
&=\left(\frac{(b(x)-d(x))\alpha(y,y)-(b(y)-d(y))\alpha(x,y)}{\alpha(x,x)\alpha(y,y)-\alpha(x,y)\alpha(y,x)},\frac{(b(y)-d(y))\alpha(x,x)-(b(x)-d(x))\alpha(y,x)}{\alpha(x,x)\alpha(y,y)-\alpha(x,y)\alpha(y,x)}\right).\nonumber
\end{align}
However due to Assumptions \ref{hyp:taux}, either $\widetilde{n}^x$ or $\widetilde{n}^y$ is negative and hence $(\widetilde{n}^x,\widetilde{n}^y)$ cannot be in the support of $\pi^{*}$. Therefore $\pi^{*}\left(\{(\widehat{n}_x,0)\} \right) = 1$
and this proves the proposition.
\end{proof}

\begin{remark}
Note that $(0,0)$, $(\widehat{n}_x,0)$, $(0,\widehat{n}_y)$ and $(\widetilde{n}_x,\widetilde{n}_y)$ are the stationary solutions of the following ordinary differential equation that approximates a large population with trait $x$ and $y$:
\begin{align}
\frac{dn_x}{dt}=& n_x(t) \Big(b(x) - d(x) - \alpha(x,x)n_x(t) - \alpha(x,y) n_y(t)\Big)\nonumber\\
\frac{dn_y}{dt}=& n_y(t) \Big(b(y) - d(y) - \alpha(y,x)n_x(t) - \alpha(y,y) n_y(t)\Big).\label{ode}
\end{align}
\end{remark}

Heuristically, the ``invasion implies substitution'' assumption prevents two traits from coexisting in the long run.
If we know which of the trait fixates, Proposition \ref{prop:identification} provides the form of the solution $\pi$ to \eqref{equation:B_pi}. In this case, we can deduce the convergence of the occupation measure of $Y^K(\cdot/Ku_K)$.

\begin{corollary}
\label{corr:identification}
Let $x,y \in \X$. For each $K \in \N^*$, let $\{ Y^K(t) : t \geq 0 \}$ be a process with generator $L^K_0$ and $\textrm{supp}(Y^K(0)) = \{x,y\}$. Let $T>0$, and suppose that there exists a $\delta > 0$ such that:
\begin{align}
\label{supportprobability}
\lim_{K \to \infty} \P \left(Y^K_t\{x\} < \delta  \textrm{ for some } t \in \left[ 0 ,\frac{T}{K u_K} \right]\right) = 0.
\end{align}
Then for any $F_f\in \mathbb{F}^2_b$,
 $$\int_0^T \int_{\mathcal{M}_F(\X)}F_f(\mu)\Gamma_0^K(dt\times d\mu):=\int_{0}^{T} F_f\left( Y^K \left( \frac{t}{K u_K} \right)  \right) dt \Rightarrow T \times F_f\left( \widehat{n}_x \delta_{x} \right)$$ as $K \to \infty$.
\end{corollary}
\begin{proof}
As in part (A) of Theorem \ref{maintheorem} we can show that $\{\Gamma^K_0 : K \in \N^*\}$ is tight in the space $\mathcal{P}\left(\mathcal{M}_F \left([0,T]\times \M_F(\X) \right)\right)$. Let $\Gamma_0$ be a limit point. Then from part (C) of Theorem \ref{maintheorem} we get that
\begin{align}
\int_{0}^{T}\int_{\mathcal{M}_F(\X)} \mathbb{B} F_f (\mu)\Gamma_0(dt \times d\mu) = 0 \textrm{ for all } F_f \in \mathbb{F}^{2}_b \textrm{ a.s.,}
\end{align}
where the operator $\mathbb{B}$ is given by \eqref{populationdynamicsoperator}.

Since $\textrm{supp}(Y^K(0)) \subset \{x,y\}$ we also have that $\textrm{supp}(Y^K(t)) \subset \{x,y\}$ for all $t \geq 0$. Let
\[ \mathcal{S}_{\delta} = \left\{ \mu \in \mathcal{M}_F(\X) :  \mu\{x\}  \leq \delta \right\}.\]
Observe that $\Gamma^K_0([0,T]\times \mathcal{S}_{\delta})\leq T$ a.s. and
\begin{align*}
0\leq \E \left( T -\Gamma^K_0([0,T]\times \mathcal{S}_{\delta}) \right) & =  K u_K \E \left[ \int_{0}^{\frac{T}{K u_K}} \left(1 -  \ind_{\mathcal{S}_{\delta}}\left( Y^K(t) \right) \right)dt \right] \\
& \leq T\ \P \left( Y^K_t\{x\} < \delta  \textrm{ for some } t \in \left[ 0 ,\frac{T}{K u_K} \right]\right).
\end{align*}
Hence by \eqref{supportprobability} we get that $ \Gamma^K_0 ([0,T] \times \mathcal{S}_{\delta})) $ converges to $T$ in $L^1(d\P)$. Because $\mathcal{S}_{\delta}$ is a closed set, $\Gamma_0 \left([0,T] \times \mathcal{S}_\delta \right) =  T \textrm{ a.s.}$

Let $\pi$ be the $\mathcal{P}\left( \mathcal{M}_F(\X) \right)$-valued random variable defined by $\pi(A) = \Gamma_0([0,T] \times A)/T$ for any $A \in \mathcal{B}\left( \mathcal{M}_F(\X) \right).$ Then $\pi \left( \mathcal{S}_{\delta} \right) = 1$ and hence $\pi$ satisfies \eqref{support} almost surely. Furthermore
$\int_{\mathcal{M}_F(\X)} \mathbb{B}  F_f (\mu)\pi(d\mu) = 0$ for all $F_f \in \mathbb{F}^{2}_b$, almost surely. Therefore using Proposition \ref{prop:identification} proves this corollary.
\end{proof}

\subsubsection{Fixation probabilities}\label{subsect2}

We have seen in Corollary \ref{corr:identification} that the behavior of a dimorphic population is known provided we know which trait fixates (see \eqref{supportprobability}). Following Champagnat et al. \cite{champagnat3,champagnatferrieremeleard2}, we can answer this question by using couplings with branching processes. We consider the process $Y^K$ started with a monomorphic or dimorphic initial condition and examine the fixation probabilities in a time interval of order $1/K u_K$.\\

We begin with some notation. For any $x \in \X$ and $\epsilon > 0 $ let
\begin{align}
\label{measurevaluednbhd}
\mathcal{N}_{\epsilon}(x) = \left\{ \mu \in \mathcal{M}_F(\X) : \textrm{supp}(\mu)=\{x\} \textrm{ and } \langle \mu , 1 \rangle \in \left[ \widehat{n}_x -\epsilon, \widehat{n}_x+\epsilon \right] \right\}.
\end{align}
Let $K\in \N^*$, $x,y\in \X$. For the process $Y^K$ of Definition \ref{def:YKL0} with initial condition $Y^K(0)=(z_1^K/K)\delta_x+(z_2^K/K)\delta_y$,  we define the $\N$-valued processes $N^K_1$ and $N^K_2$ by
$N^K_1(t) = K \langle  Y^K(t) ,\ind_{\{x\}} \rangle$ and $N^K_2(t) = K \langle Y^K(t), \ind_{\{y\}} \rangle$. Then $\{N^K(t) = (N^K_1(t),N^K_2(t)) : t \geq 0\}$ is the $\N \times \N$-valued Markov jump process with transition rates given by
\begin{equation}\begin{array}{cc}
m b_1 := m b(x)  & \textrm{ from } (m,n) \textrm{ to }  (m+1,n) \\
n b_2 := n b(y)  & \textrm{ from } (m,n) \textrm{ to }  (m,n+1) \\
m\,d_1(m,n):=m \left( d(x) + \alpha(x,x) \frac{m}{K} + \alpha(x,y) \frac{n}{K}  \right) &  \textrm{ from } (m,n) \textrm{ to }  (m-1,n) \\
n\, d_2(m,n):=n \left( d(y) + \alpha(y,x) \frac{m}{K} + \alpha(y,y) \frac{n}{K}  \right) &  \textrm{ from } (m,n) \textrm{ to }  (m,n-1) .
\end{array}
\label{transitions}\end{equation}
and with initial state $(N^K_1(0),N^K_2(0)) = (z^K_1, z^K_2)$. Notice that $d_1$, $d_2$ are increasing functions in $n$ and $m$.\\
Let also $\mathcal{S}_K = [A_K,B_K) \times [C_K,D_K)$ be a subset of $\R_{+}^2$ and define
\begin{align}
\label{extitimeforsk}
T_{\mathcal{S}_K} = \inf \left\{ t \geq 0 :  N^K(t) \notin \mathcal{S}_K \right\}.
\end{align}
The results at the end of this section require comparisons with branching processes, which we recall in the next proposition whose proof is provided in the Appendix.

\begin{proposition}
\label{proposition:couplingwithbranchingprocesses2}
For each $K \in \N^*$ let $A_K$, $B_K$, $C_K$, $D_K\in \R_+$ and let $\mathcal{S}_K$ and $(N^K_1,N^K_2)$ be defined as above, with $N^K(0) \in \mathcal{S}_K$.
For parts (A) and (B), assume that $\eta >0$, $\epsilon \in (0,1)$ are such that $[K\eta(1-\epsilon),K\eta(1+\epsilon)] \subset [A_K,B_K)$. Let
\begin{align*}
d^{-}_1 (\eta,\epsilon) =  & \inf_{ K \in \N^*}  \inf\{ d_1(m,n)\ :\ (m,n)\in [K\eta(1-\epsilon),K\eta(1+\epsilon)] \times [C_K,D_K) \}\\
d^{+}_1 (\eta,\epsilon) =  & \sup_{K \in \N^*}   \sup\{ d_1(m,n)\ :\ (m,n)\in [K\eta(1-\epsilon),K\eta(1+\epsilon)] \times [C_K,D_K) \}.
\end{align*}
\begin{itemize}
\item[(A)] If $b_1 < d^{-}_1(\eta,\epsilon)$ and $N^K_1(0) \leq K \eta$ then $N^K_1$ exists $[A_K,B_K)$ by $A_K$ a.s.: for any $T>0$
\begin{align*}
\lim_{K \to \infty} \P \left( T_{\mathcal{S}_K} \leq \frac{T}{K u_K} , N^K_1 \left( T_{\mathcal{S}_K} \right) \geq B_K \right) = 0.
\end{align*}
\item[(B)] If $b_1 > d^{+}_1(\eta,\epsilon)$ and $N^K_1(0) \geq K\eta$ then $N^K_1$ exists $[A_K,B_K)$ by $B_K$ a.s.: for any $T>0$
\begin{align*}
\lim_{K \to \infty} \P \left( T_{\mathcal{S}_K} \leq \frac{T}{K u_K} , N^K_1 \left( T_{\mathcal{S}_K} \right) < A_K \right) = 0.
\end{align*}
\end{itemize}
Now assume that there exist $\epsilon_1,\epsilon_2 \in (0,1)$ and $\eta_1 < \eta_2$ such that $[K\eta_i(1-\epsilon_i),K\eta_i(1+\epsilon_i)] \subset [A_K,B_K)$ for $i=1,2$ and $d^{+}_1 (\eta_1,\epsilon_1)< b_1 < d^{-}_1 (\eta_2,\epsilon_2)$. Also suppose that $N^K_1(0) \in [K \eta_1,K \eta_2]$.
\begin{itemize}
\item[(C)] Then for any $T>0$
\begin{align*}
\lim_{K \to \infty} \P \left( T_{\mathcal{S}_K} \leq \frac{T}{K u_K} , N^K_1 \left( T_{\mathcal{S}_K} \right) \notin [A_K,B_K) \right) = 0.
\end{align*}
\end{itemize}
Let
$[C_K,D_K) = [1, K\epsilon )$ for some $\epsilon>0$ and define $d^{-}_2 (\epsilon),d^{+}_2 (\epsilon)$
\begin{multline*}
d^{-}_2 (\epsilon):= \inf_{K \in \N^*}\inf \{d_2(m,n) \ : \ (m,n)\in  [A_K,B_K) \times [1, K\epsilon ) \} \\
\leq \sup_{K \in \N^*}\sup\{d_2(m,n)\ : \ (m,n)\in [A_K,B_K)\times [1, K\epsilon ) \} := d^{+}_2(\epsilon).
\end{multline*}
Let $t_K$ be any $\N^{*}$-valued sequence such that $\log K \ll t_K \ll \frac{1}{Ku_K}$
\begin{itemize}
\item[(D)] If $b_2 < d^{-}_2(\epsilon)$ and $N^K_2(0) \leq K \epsilon/2$ then
\begin{align*}
\lim_{K \to \infty} \P \left( T_{\mathcal{S}_K} \leq t_K , N^K_2 \left( T_{\mathcal{S}_K} \right) = 0 \right) = 1.
\end{align*}
\item[(E)] If $b_2 > d^{+}_2(\epsilon)$ and $N^K_2(0) = 1$ then
\begin{align*}
 1 - \frac{d^{+}_2(\epsilon)}{b_2} \leq & \lim_{K \to \infty} \P \left( T_{\mathcal{S}_K} \leq t_K , N^K_2 \left( T_{\mathcal{S}_K} \right) \geq \epsilon K \right) \leq 1 - \frac{d^{-}_2(\epsilon)}{b_2}\\
\mbox{ and }\qquad  \frac{d^{-}_2(\epsilon)}{b_2} \leq & \lim_{K \to \infty} \P \left( T_{\mathcal{S}_K} \leq t_K , N^K_2 \left( T_{\mathcal{S}_K} \right) = 0 \right) \leq  \frac{d^{+}_2(\epsilon)}{b_2}.
\end{align*}
\end{itemize}
\end{proposition}

Using the results of Proposition \ref{proposition:couplingwithbranchingprocesses2}, we can deduce the following result for the fixation probabilities for a branching process $Y^K$ with a dimorphic initial condition and for the behavior at times of order $1/Ku_K$.
\begin{proposition}
\label{main_proposition}
For each $K \in \N^*$ let $\{Y^K(t) : t \geq 0\}$ be a process with generator $L^K_0$. Pick two traits $x,y \in \X$ initially present in the population and two $\N^*$-valued sequences $\{z^{K}_{1} : K \in \N^*\}$ and $\{z^{K}_{2} : K \in \N^* \}$ that give the sizes of the populations of trait $x$ and $y$. Assume that
\[ Y^K(0) = \frac{z^K_1}{K} \delta_{x} + \frac{z^K_2}{K} \delta_{y}.\]
Then we have the following for any $T>0$.
\begin{itemize}
\item[(A)] A monomorphic population with a sufficiently large size does not die:\\
Suppose that for some $\epsilon>0$, $z^K_1 \geq K \epsilon$ and $z^K_2 = 0$ for each $K \in \N^*$. Then for some $\delta > 0$
\begin{align}
\label{main_proposition_parta}
 \lim_{K \to \infty} \P \left( \exists t \in \left[ 0 , \frac{T}{K u_K} \right],\  Y^K_t\{x\} < \delta \right) = 0.
\end{align}
\item[(B)] A monomorphic population with a size around $\widehat{n}_x$ remains there: \\
Suppose that for some $\epsilon>0$, $z^K_1 \in [ K(\widehat{n}_x -\epsilon), K(\widehat{n}_x+\epsilon)]$ and $z^K_2 = 0$ for each $K \in \N^*$. Then
\begin{align}
\label{main_proposition_parta}
 \lim_{K \to \infty} \P \left(\exists t \in \left[ 0 , \frac{T}{K u_K} \right],\ Y^K(t) \notin \mathcal{N}_{2 \epsilon}(x)\right) = 0.
\end{align}
\item[(C)] A favorable mutant with a non-negligible size in a resident population near equilibrium fixates a.s.:\\
Suppose that $\mathrm{Fit}(y,x)> 0 $ and for some $\epsilon>0$, $z^K_1 < K (\widehat{n}_x+ \epsilon)$ and $z^K_2 >  K \epsilon $ for all $K \in \N^*$. Then there exists an $\epsilon_0>0$ such that if $\epsilon <\epsilon_0$ then
\begin{align}
\label{main_proposition_partb}
 \lim_{K \to \infty} \P \left(\exists t \in \left[ 0 , \frac{T}{K u_K} \right],\ Y^K_t\{y\}    < \frac{\epsilon}{2}  \right) = 0.
\end{align}
\end{itemize}
\noindent For parts (D) and (E), we consider a small mutant population in a resident population near its equilibrium. We assume that for some $\epsilon > 0$ we have $z^{K}_1 \in [ K(\widehat{n}_x -\epsilon), K(\widehat{n}_x+\epsilon)]$ and $z^K_2 \in [1, K \epsilon)$ for all $K \in \N^*$. Let $t_K$ be any $\N^{*}$-valued sequence such that $\log K \ll t_K \ll 1/K u_K$. Let $\mathcal{S}_K=[K(\widehat{n}_x-2\epsilon),K( \widehat{n}_x+2\epsilon )\times [1,2K \epsilon)$ and let $T_{\mathcal{S}_K}$ be the associated stopping time \eqref{extitimeforsk}.
\begin{itemize}
\item[(D)] An unfavorable mutant dies out in a time $t_K$. \\
Let $\mathrm{Fit}(y,x) < 0$. There exists an $\epsilon_0 > 0$ such that if $\epsilon < \epsilon_0$ then
\begin{align}
\label{main_proposition_partc}
\lim_{K \to \infty} \P \left(  T_{\mathcal{S}_K} \leq t_K , Y^K_{ T_{\mathcal{S}_K}}\{y\} = 0 \right) = 1.
\end{align}
\item[(E)] A favorable mutant invades with probability $\mathrm{Fit}(y,x)/b(y)$. \\
Let $\mathrm{Fit}(y,x)>0$ and $z^{K}_{2} = 1$ for all $K \in \N^*$. Then there exist positive constants $c,\epsilon_0$ such that for all $\epsilon < \epsilon_0$ we have
\begin{align}
\label{main_proposition_partd1}
& \lim_{K \to \infty} \left|\P \left(   T_{\mathcal{S}_K} \leq t_K , Y^K_{T_K(\epsilon)}\{y\} \geq 2 \epsilon \right)- \frac{ \mathrm{Fit}(y,x) }{b(y)} \right|\leq c \epsilon,\\
\label{main_proposition_partd2}
\mbox{and }\qquad & \lim_{K \to \infty} \left| \P \left(  T_{\mathcal{S}_K} \leq t_K , Y^K_{T_K(\epsilon)}\{y\} = 0 \right) - \left( 1- \frac{\mathrm{Fit}(y,x)}{b(y)}\right) \right|\leq  c \epsilon.
\end{align}

\end{itemize}
\end{proposition}

\begin{proof}
Let us prove part (A). Since $z^K_2 = 0$, $N^K_2(t) = 0$ for all $t\geq 0$ and $K \in \N^*$. Note that $b(x)-d(x) >0$ and so we can pick a $\delta >0$ such that $b(x)- d(x) - 2 \alpha(x,x) \delta > 0$ and $2 \delta < \epsilon$. Let $\mathcal{S}'_K = [K \delta , \infty) \times \{0\}$ and let $T_{\mathcal{S}'_K}$ be given by \eqref{extitimeforsk}. Then
\begin{align*}
 \lim_{K \to \infty} \P\left( Y^K_t\{x\}   < \delta \textrm{ for some } t \in \left[ 0 , \frac{T}{K u_K} \right]\right)
&\leq \lim_{K \to \infty} \P \left( T_{\mathcal{S}'_K} \leq \frac{T}{K u_K}\right)\\
& = \lim_{K \to \infty} \P \left( T_{\mathcal{S}'_K} \leq \frac{T}{K u_K} , N^K_1\left( T_{\mathcal{S}'_K} \right) < K \delta\right).
\end{align*}
The last equality above follows from the fact that the only way to exit the set $\mathcal{S}'_K$ is by having $N^K_1$ go below $K \delta$.
Observe that on the set $[K \delta, 2 K \delta] \times \{0\}$, the supremum of $d_1$ is bounded above by $d(x) + 2 \alpha(x,x) \delta$ which is less than $b(x)$. Therefore by part (B) of Proposition \ref{proposition:couplingwithbranchingprocesses2}:
\begin{align*}
\lim_{K \to \infty} \P \left( T_{\mathcal{S}'_K} \leq \frac{T}{K u_K} , N^K_1\left( T_{\mathcal{S}'_K} \right) < K \delta\right) = 0,
\end{align*}
which proves part (A).\\

For part (B) we can choose $\epsilon$ sufficiently small such that $\widehat{n}_x > 2 \epsilon$. Let $\eta_1 = \widehat{n}_x - \epsilon$ and $\eta_2 = \widehat{n}_x + \epsilon$. We can find $\epsilon_1,\epsilon_2>0$ such that $[\eta_1(1-\epsilon_1), \eta_1(1+\epsilon_1)] \subset [\widehat{n}_x - 2\epsilon, \widehat{n}_x)$, $[\eta_2(1-\epsilon_2), \eta_2(1+\epsilon_2)] \subset (\widehat{n}_x, \widehat{n}_x + 2\epsilon]$ and on the set $[K \eta_1(1-\epsilon_1), K \eta_1 (1+\epsilon_1)] \times \{0\}$ the supremum of $d_1$ is strictly below $b(x)$ while on the $[K \eta_2(1-\epsilon_2), K \eta_2 (1+\epsilon_2)] \times \{0\}$ the infimum of $d_1$ is strictly above $b(x)$. Let $\mathcal{S}'_K = [K (\widehat{n}_x - 2\epsilon) , K(\widehat{n}_x + 2\epsilon)) \times \{0\}$ and let $T_{\mathcal{S}'_K}$ be given by \eqref{extitimeforsk}. From part (B) of Proposition \ref{proposition:couplingwithbranchingprocesses2} we get
\begin{align*}
\lim_{K \to \infty} \P \left( T_{\mathcal{S}'_K} \leq \frac{T}{K u_K} , N^K_1\left( T_{\mathcal{S}'_K} \right) \notin [K (\widehat{n}_x - 2\epsilon) , K(\widehat{n}_x + 2\epsilon)) \right) = 0.
\end{align*}
Observe that $\textrm{supp}(Y^K(t)) = \{x\}$ for all $t \geq 0$. Hence this limit proves part (B).\\

For part (C) note that $\mathrm{Fit}(y,x)>0$. We can choose $\epsilon_0>0$ such that $d(y) + \alpha(y,x) (\widehat{n}_x + 2 \epsilon_0 ) + 2 \alpha(y,y)\epsilon_0 = b(y)$. Now let $\epsilon < \epsilon_0$ and assume that $z^K_1 < K (\widehat{n}_x+\epsilon)$ and $z^K_2 >  K \epsilon $ for all $K \in \N^*$, as stated in the proposition. Define the set $\mathcal{S}'_K = [ 0 , K (\widehat{n}_x + 2 \epsilon) ) \times [K \epsilon /2,\infty)$ and let $T_{\mathcal{S}'_K}$ be given by \eqref{extitimeforsk}. It is easy to see that
\begin{multline*}
 \lim_{K \to \infty} \P\left(   Y^K_t\{y\}  < \frac{\epsilon}{2}  \textrm{ for some } t \in \left[ 0 , \frac{T}{K u_K} \right] \right)
\\
\leq \lim_{K \to \infty}  \Big[  \P \Big( T_{\mathcal{S}'_K} \leq \frac{T}{K u_K}, N^K_1\left( T_{\mathcal{S}'_K} \right) \geq K (\widehat{n}_x + 2 \epsilon)\Big)
+ \P \Big( T_{\mathcal{S}'_K} \leq \frac{T}{K u_K},N^K_2\left( T_{\mathcal{S}'_K} \right)<  \frac{K\epsilon}{2}\Big)\Big].
\end{multline*}
On the set $[K (\widehat{n}_x + \epsilon/2), K (\widehat{n}_x + 3\epsilon/2) ] \times [K \epsilon/2,\infty)$, the infimum of $d_1$ is greater that $b(x)$. Part (A) of Proposition \ref{proposition:couplingwithbranchingprocesses2} shows that the first limit on the right is $0$.
On the set $[ 0 , K (\widehat{n}_x + 2 \epsilon) ) \times [K \epsilon /2 , 3 K \epsilon /4 )$, the supremum of $d_2$ is less than $b(y)$. We can use part(B) of Proposition \ref{proposition:couplingwithbranchingprocesses2}, to see that the second limit on the right is also $0$. This proves part (C).\\

For part (D), observe that $\mathrm{Fit}(y,x)<0$ and let $\epsilon_0>0$ satisfy $d(y)+ \alpha(y,x) (\widehat{n}_x- 2\epsilon_0) =  b(y)$. Pick an $\epsilon \in (0,\epsilon_0)$. On the set $[ K(\widehat{n}_x -2 \epsilon), K(\widehat{n}_x+ 2 \epsilon)] \times [1, 2 K \epsilon)$ the infimum of $d_2$ is greater than $b(y)$. Part (D) of Proposition \ref{proposition:couplingwithbranchingprocesses2} proves part (D).\\

For part (E), note that $\mathrm{Fit}(y,x)>0$ and let $\epsilon_0>0$ satisfy $d(y)+ \alpha(y,x) (\widehat{n}_x+ 2 \epsilon_0) + 2 \alpha(y,y) \epsilon_0  = b(y)$. Pick $\epsilon\in (0,\epsilon_0)$. On the set $[ K(\widehat{n}_x -2 \epsilon), K(\widehat{n}_x+ 2 \epsilon)] \times [1, 2 K \epsilon)$ the supremum of $d_2$ is less than $d^{+}_2(\epsilon) := d(y)+ \alpha(y,x) (\widehat{n}_x+ 2 \epsilon) + 2 \alpha(y,y) \epsilon$ and the infimum of $d_2$ is greater than $d^{-}_2(\epsilon) :=  d(y)+ \alpha(y,x) (\widehat{n}_x- 2 \epsilon)$. Both $d^{+}_2(\epsilon)$ and $d^{-}_2(\epsilon)$ are less than $b(y)$. Using part (E) of Proposition \ref{proposition:couplingwithbranchingprocesses2} proves part (E).
\end{proof}

From the preceding proposition, we can retrieve the state of the process on a large window $[\epsilon/Ku_K,\epsilon^{-1}/Ku_K]$ given the initial condition. This allows us to understand what will happen if we neglect the transitions and large time rare events.

\begin{corollary}
\label{maincorollary}
Let us consider the process $Y^K$ of Definition \ref{def:YKL0}.
\begin{itemize}
\item[(A)] Suppose that for some $x \in \X$, $\textrm{supp}(Y^K(0)) = \{x\}$ and $Y^K_0\{x\}>\epsilon$ for all $K \in \N^*$. Then
 \begin{equation} \lim_{K \to \infty} \P \left(  Y^K(t) \in  \mathcal{N}_{\epsilon}(x) \textrm{ for all } t \in \left[\frac{\epsilon}{K u_K}, \frac{\epsilon^{-1}}{Ku_K} \right]  \right) = 1.\label{desiredinequalityforE1}\end{equation}

\item[(B)] Suppose that for some $x,y \in \X$ such that $\mathrm{Fit}(y,x) < 0$, we have $\textrm{supp}(Y^K(0)) = \{x,y\}$ with $Y^K_0\{x\}\in  \left[ \widehat{n}_x -\epsilon, \widehat{n}_x+\epsilon \right]$ and $Y^K_0\{y\}<\epsilon$ for all $K \in \N^*$. Then
 \begin{equation*}\lim_{K \to \infty} \P \left(  Y^K(t) \in  \mathcal{N}_{\epsilon}(x) \textrm{ for all } t \in \left[\frac{\epsilon}{K u_K}, \frac{\epsilon^{-1}}{Ku_K} \right] \right) =  1.
\end{equation*}
\item[(C)]  Suppose that for some $x,y \in \X$ such that $\mathrm{Fit}(y,x) > 0$, we have $\textrm{supp}(Y^K(0)) = \{x,y\}$ with $Y^K_0\{x\}  \in \left[ \widehat{n}_x -\epsilon, \widehat{n}_x+\epsilon \right]$ and $ Y^K_0\{y\} = 1/K$ for all $K \in \N^*$. Then
\begin{align*}
& \lim_{\epsilon \to 0} \lim_{K \to \infty} \P \left(  Y^K(t) \in  \mathcal{N}_{\epsilon}(x) \textrm{ for all } t \in \left[\frac{\epsilon}{K u_K}, \frac{\epsilon^{-1}}{Ku_K} \right] \right) \\
&= 1-  \lim_{\epsilon \to 0} \lim_{K \to \infty} \P \left(  Y^K(t) \in  \mathcal{N}_{\epsilon}(y) \textrm{ for all } t \in \left[\frac{\epsilon}{K u_K}, \frac{\epsilon^{-1}}{Ku_K} \right] \right)\\
& =  1 - \frac{ \mathrm{Fit}(y,x)}{b(y)}.
\end{align*}
\end{itemize}

\end{corollary}

\begin{proof} Let us first consider part (A). Let $\epsilon > 0$ and $Y^K$ be the process of Definition \ref{def:YKL0}, with the initial condition stated in the statement of part (A). Proposition \ref{main_proposition} (A) implies that for some $\delta>0$
\begin{align*}
\lim_{K \to \infty} \P \left( Y^K_t\{x\} < \delta  \textrm{ for some } t \in \left[ 0 ,\frac{\epsilon^{-1}}{K u_K} \right]\right) = 0.
\end{align*}
From Corollary \ref{corr:identification} we know that for any $t \geq 0$ and $F_f\in \mathbb{F}^2_b$,
\[ \int_{0}^{t} \mathbb{F}_f \left( Y^K \left( \frac{s}{K u_K} \right) \right) ds \Rightarrow t F_f\left( \widehat{n}_x \delta_{x} \right)\]
 as $K\rightarrow +\infty$.
Hence if we define $\sigma^K_{\epsilon} = \inf \{ t \geq 0 : Y^K(t) \in  \mathcal{N}_{\epsilon/2}(x)\}$, then $K u_K\sigma^K_{\epsilon} \rightarrow 0$ in probability as $K \to \infty$. Now let the process $\{\tilde{Y}^K(t): t \geq 0\}$ be given by $\tilde{Y}^K_{\epsilon}(t) = Y^K_{\epsilon}(t+\sigma^K_{\epsilon})$.  By the strong Markov property, this process also has generator $L^K_0$. Moreover its initial state is inside $\mathcal{N}_{\epsilon/2}(x)$. Using part (B) of Proposition \ref{main_proposition} proves part (A).\\

For part (B), fix an $\epsilon>0$ and consider the process $Y^K$ with the initial condition specified in the statement. Let us consider the stopping time $T_{\mathcal{S}_K}$ associated with $\mathcal{S}_K=[ K(\widehat{n}_x - 2\epsilon), K(\widehat{n}_x + 2\epsilon))\times [K \epsilon,2K\epsilon)$ by \eqref{extitimeforsk} and a sequence $\{t_K\}$ as in Proposition \ref{main_proposition}. Since $\mathrm{Fit}(y,x) < 0$, thanks to Proposition \ref{main_proposition} (D):
\begin{multline}
 \lim_{K \to \infty} \P \left(  Y^K(t) \in  \mathcal{N}_{\epsilon}(x) \textrm{ for all } t \in \left[\frac{\epsilon}{K u_K}, \frac{\epsilon^{-1}}{Ku_K} \right] \right) \\
=  \lim_{K \to \infty} \P \left(  Y^K(t) \in  \mathcal{N}_{\epsilon}(x) \textrm{ for all } t \in \left[\frac{\epsilon}{K u_K}, \frac{\epsilon^{-1}}{Ku_K} \right] \ ; \ T_{\mathcal{S}_K}\leq t_K\ ; \ Y^K_{T_{\mathcal{S}_K}}\{y\}=0\right).\label{etape1}
\end{multline}
Let $\{ \tilde{Y}^K(t) : t \geq 0\}$ be the process given by $\tilde{Y}^K(t) = Y^K\left(T_{\mathcal{S}_{T_K}} + t  \right)$. By the strong Markov property, this process also has generator $L^K_{0}$. On the event $\{T_{\mathcal{S}_K}\leq t_K\ ; \ Y^K_{T_{\mathcal{S}_K}}\{y\}=0\}$, $\tilde{Y}^K(0)$ is such that $\supp(\tilde{Y}^K(0))=\{x\}$ and $\tilde{Y}^K_0\{x\} \in   [ \widehat{n}_x - 2\epsilon, \widehat{n}_x + 2\epsilon)$. Applying Part (A) of the corollary for $\tilde{Y}^K$ provides:
\begin{align*}
   \lim_{K \to \infty} \P \left(  \tilde{Y}^K(t)  \in  \mathcal{N}_{\epsilon}(x) \textrm{ for all } t \in \left[\frac{\epsilon}{Ku_K} , \frac{\epsilon^{-1}}{Ku_K} \right] \ ; \ T_{\mathcal{S}_K}\leq t_K\ ; \ Y^K_{T_{\mathcal{S}_K}}\{y\}=0\right) = 1.
\end{align*}
Moreover, since $K u_K t_K \to 0$, we obtain \eqref{desiredinequalityforE1} when $T_{\mathcal{S}_K}<t_K$, which proves part (B).\\


For part (C), fix an $\epsilon>0$ and define $Y^K$ with the initial condition specified in the statement. Let us consider $\mathcal{S}_K$ and $T_{\mathcal{S}_K}$ as in part (B).
  We can write
\begin{align}
&\P \left(  Y^K(t) \in  \mathcal{N}_{\epsilon}(x) \textrm{ for all } t \in \left[\frac{\epsilon}{K u_K}, \frac{\epsilon^{-1}}{Ku_K} \right] \right) \nonumber\\
&= \sum_{i=1}^3 \P \left(  Y^K(t) \in  \mathcal{N}_{\epsilon}(x) \textrm{ for all } t \in \left[\frac{\epsilon}{K u_K}, \frac{\epsilon^{-1}}{Ku_K} \right] \ ; \  E^K_i(\epsilon)\right),\label{etape4}
\end{align}where $E^K_1(\epsilon)=\left\{ T_{\mathcal{S}_K}  \leq t_K ,   Y^K_{T_{\mathcal{S}_K}}\{y\} = 0 \right\}$, $E^K_2(\epsilon)=\left\{ T_{\mathcal{S}_K}  \leq t_K ,  Y^K_{T_{\mathcal{S}_K}}\{y\} \geq 2\epsilon \right\}$ and $E^K_3(\epsilon)=  \left(  E^K_1(\epsilon)\cup  E^K_2(\epsilon)\right)^{c}$.\\
Let us consider the term in \eqref{etape4} corresponding to $i=1$. Since $t_K<\epsilon/Ku_K$ for sufficiently large $K$ we have
\begin{multline}
\P \left(  Y^K(t) \in  \mathcal{N}_{\epsilon}(x) \textrm{ for all } t \in \left[\frac{\epsilon}{K u_K}, \frac{\epsilon^{-1}}{Ku_K} \right] \ ; \  E^K_1(\epsilon)\right)\\
\begin{aligned}
= & \E \left(  \ind_{\{T_{\mathcal{S}_K} \leq t_K ; Y^K_{T_{\mathcal{S}_K}}\{y\}=0\}} \P\left( Y^K(t) \in  \mathcal{N}_{\epsilon}(x) \textrm{ for all } t \in \left[\frac{\epsilon}{K u_K}, \frac{\epsilon^{-1}}{Ku_K} \right]  \middle\vert   \mathcal{F}_{T_{\mathcal{S}_K}}\right) \right)\\
\end{aligned}\label{etape5}
\end{multline}On the event $E_1^K(\epsilon)$, we also have $Y_{T_{\mathcal{S}_K}}\{x\}\in [\widehat{n}_x-2\epsilon,\widehat{n}_x+2\epsilon)$. Thus, applying the strong Markov property, part (A) and the fact that $Ku_K t_K\rightarrow 0$,
the probability inside the expectation in the r.h.s. of \eqref{etape5} converges to 1. From part (E) of Proposition \ref{main_proposition}, the term in \eqref{etape4} corresponding to $i=1$ converges to $1-\Fit(y,x)/b(y)$ as $K \to \infty$ and $\epsilon \to 0$.\\
For the term corresponding to $i=2$, we condition in a way similar to \eqref{etape5}. On the event $E_2^K(\epsilon)$, $Y^K_{T_{\mathcal{S}_K}}\{y\}\geq 2\epsilon$ and $Y^K_{T_{\mathcal{S}_K}}\{x\}\in [ \widehat{n}_x - 2\epsilon, \widehat{n}_x + 2\epsilon)$. From part (C) of Proposition \ref{main_proposition}, the probability of the process $Y^K_{\cdot}\{y\}$ going below $\epsilon$ between $T_{\mathcal{S}_K}$ and $T_{\mathcal{S}_K}+\epsilon^{-1}/K u_K$ tends to 0 when $K\rightarrow +\infty$. Hence, the condition \eqref{supportprobability} of Corollary \ref{corr:identification} is satisfied and as a consequence, the stopping time:
$$\sigma_{K,\epsilon}=\inf\big\{ t\geq T_{\mathcal{S}_K} \ : \ \tilde{Y}^K_t\{x\}<\epsilon \mbox{ and }\tilde{Y}^K_t\{y\}\in [\widehat{n}_y -\epsilon, \widehat{n}_y+\epsilon]\big\}$$satisfies $Ku_K\sigma_{K,\epsilon}\rightarrow 0$ in probability.
Conditioning by $\mathcal{F}_{\sigma_{K,\epsilon}}$, and using part (B), we show that the term corresponding to $i=2$ in \eqref{etape4} converges to 0. \\
The term for $i=3$ converges to 0 since $\lim_{\epsilon \to 0 }\lim_{K \to \infty} \P \left( E^K_3(\epsilon)\right) = 0$.\\
Gathering the results for $i\in \{1,2,3\}$, we easily get
\[\lim_{\epsilon \to 0 }\lim_{K \to \infty}  \P \left(  Y^K(t) \in  \mathcal{N}_{\epsilon}(x) \textrm{ for all } t \in \left[\frac{\epsilon}{K u_K}, \frac{\epsilon^{-1}}{Ku_K} \right] \right) = 1- \frac{\mathrm{Fit}(y,x)}{b(y)}. \]
The proof that $\lim_{\epsilon \to 0 }\lim_{K \to \infty}  \P \left(  Y^K(t) \in  \mathcal{N}_{\epsilon}(y) \textrm{ for all } t \in \left[\frac{\epsilon}{K u_K}, \frac{\epsilon^{-1}}{Ku_K} \right] \right) =  \mathrm{Fit}(y,x)/b(y)$ is similar. This completes the proof of part (C) of the corollary.
\end{proof}

\subsection{Proof of Theorem \ref{theorem:convergenceTSS}: convergence to the TSS}

We have now the tools to prove Theorem \ref{theorem:convergenceTSS}. By Theorem \ref{maintheorem}, the distributions of $\{(\chi^K,\Gamma^K) : K \in \N^*\}$ are tight. Let $(\chi,\Gamma)$ be a limiting value satisfying \eqref{limitingmartingale1} and \eqref{limitingmartingale2}. Using Prohorov's theorem, there exists a subsequence $\{(\tilde{\chi}^K,\tilde{\Gamma}^K)\}$ that converges in distribution to $(\chi,\Gamma)$. By Skorokhod representation theorem (see e.g. \cite{billingsley}), there exists on the same probability space as $(\chi,\Gamma)$ a sequence again denoted by $\{(\chi^K,\Gamma^K)\}$ with an abuse of notation, that converges a.s. to $(\chi,\Gamma)$ and which has the same marginal distributions as $\{(\tilde{\chi}^K,\tilde{\Gamma}^K)\}$. Observe that the process $\{\chi'(t) : t \geq 0\}$ (introduced in the statement of Theorem \ref{theorem:convergenceTSS}) uniquely determines $\Gamma$ (through \eqref{occupationmeasurerelation}) and $\Gamma$ uniquely determines the process $\{\chi(t) : t \geq 0\}$ (through the martingale problem given by \eqref{limitingmartingale1}). It thus remains to prove part (B) of Theorem \ref{theorem:convergenceTSS}.\\
The main idea to identify the limiting values is that between subsequent appearances of new mutants, our process $X^K$ behaves like the process considered in Corollary \ref{maincorollary}. When a fit mutant appears, it either gets eradicated quickly or
the process stabilizes around the new monomorphic equilibrium characterized by the mutant trait. Between two rare mutations, the trait and size of the monomorphic equilibrium can be inferred from the occupation measure, because the population is monomorphic and because its size is shown to reach an equilibrium.\\

For $K\in \N^*,\ i \in \N^{*}$ let $\tau^K_i$ and $\tau_i$ be the $i$-th jump times of the process $\chi^K$ and $\chi$ respectively. For convenience we define $\tau^K_{0}= \tau_{0} =0$. Since $\left( \chi^K, \Gamma^K \right) \rightarrow \left(\chi,\Gamma \right)$ a.s., then for any $i \in \N^*$, $\left(\tau^K_{i-1},\tau^K_i \right) \rightarrow \left(\tau_{i-1} , \tau_i\right)$ a.s. Using \eqref{limitingmartingale1} and Lemma \ref{lemma:boundednessofmass}, we have that almost surely $\tau_i-\tau_{i-1}\in (0,+\infty)$ for each $i \in \N^*$. Thus
\begin{align}
\label{timebetweenmutations}
\lim_{\epsilon\rightarrow 0}\lim_{K \to \infty} \P \left( \tau^K_{i} - \tau^K_{i-1} \in [ \epsilon , \epsilon^{-1} ]  \right) = 1.
\end{align}
In the rest of the proof we will assume that we are always on the set $\{\forall i\in \N^*, \tau^K_i-\tau^K_{i-1}\geq \epsilon\}$. \\

For $\epsilon>0$ and $t>0$ define:
\begin{align*}
& R^{K,\epsilon}_i(t)  = \ind_{\{\tau_i^K-\tau^K_{i-1}\geq \epsilon \}} \frac{\int_{(\tau^K_{i-1}+\epsilon)\wedge t}^{\tau^K_i \wedge t} \int_{\mathcal{M}_F(\X)}\langle \mu, x\rangle \Gamma^K(ds \times d \mu)}{\int_{(\tau^K_{i-1}+\epsilon)\wedge t}^{\tau^K_i \wedge t} \int_{\mathcal{M}_F(\X)}\langle \mu, 1\rangle \Gamma^K(ds \times d \mu)}\\
& R_i (t)= \frac{\int_{\tau_{i-1}\wedge t}^{\tau_i \wedge t} \int_{\mathcal{M}_F(\X)}\langle \mu, x \rangle \Gamma(dt \times d \mu)}{\int_{\tau_{i-1}\wedge t}^{\tau_i\wedge t} \int_{\mathcal{M}_F(\X)}\langle \mu, 1\rangle \Gamma(dt \times d \mu)}.
\end{align*}Note that if $\supp( Z^K(s) ) = \{x\}$ for all $s \in [  \tau^K_{i-1}\wedge t + \epsilon, \tau^K_i\wedge t  )$, then $R^{K,\epsilon}_i(t) = x$, and the same holds for $Z$ and $R_i$. Heuristically, $R^{K,\epsilon}_i$ and $R_i$ are the estimators of the trait of the monomorphic population that fixes between the $(i-1)^{\mbox{\scriptsize{th}}}$ and the $i^{\mbox{\scriptsize{th}}}$ mutations.

Let us define:
\begin{align*}
\chi'^{K,\epsilon}(t)=x_0\ind_{t<\epsilon}+\sum_{i=1}^{+\infty}R^{K,\epsilon}_i(t) \ind_{[\tau^K_{i-1}+\epsilon,\tau_i^K+\epsilon)}(t) \qquad \mbox{ and }\qquad \chi'(t)=\sum_{i=1}^{+\infty}R_i(t) \ind_{[\tau_{i-1},\tau_i)}(t).
\end{align*}
The a.s. convergence $\left( \chi^K, \Gamma^K \right) \rightarrow \left(\chi,\Gamma \right)$ also implies that $\chi'^{K,\epsilon}$ converges to $\chi'$ as $K \to +\infty$ and $\epsilon \to 0$ in the Skorokhod space $\D([0,T],\X)$ for any $T>0$. \\

For any $i \in \N^*$ and $\epsilon>0$ define an event
\begin{align*}
E^K_i(\epsilon) = & \left\{  Z^K(t) \in \mathcal{N}_{\epsilon}(R^{K,\epsilon}_i(t)) \textrm{ for all } t \in \left[  \tau^K_{i-1} + \epsilon, \tau^K_i  \right)\right\},
\end{align*}
where for any $x \in \X $, $ \mathcal{N}_{\epsilon}(x)$ is defined by \eqref{measurevaluednbhd}. On $E^K_i(\epsilon)$, $R^{K,\epsilon}_i(t)=R^{K,\epsilon}_i(\tau^K_{i-1}+\epsilon)$ for all $t\in [\tau_{i-1}^K+\epsilon,\tau^K_i)$ and with an abuse of notation, we will write $R^{K,\epsilon}_i$ for this value. \\
Now, thanks to the convergence of the $\tau_i^K$'s and Lemma \ref{lemma:boundednessofmass}, the proof of part (B) of Theorem \ref{theorem:convergenceTSS} is done if we prove the following proposition:

\begin{proposition}\label{propcvfinale}
Under Assumptions \ref{hyp:taux} and \ref{assumption:monomorphicIC}, we have for all $i\in \N^*$ that:
\begin{align}
\label{mainfactfori=0}
\lim_{\epsilon \to 0 } \lim_{K \to \infty} \P \left( E^K_i(\epsilon)\right) = 1.
\end{align}
\end{proposition}

Indeed, if Proposition \ref{propcvfinale} is proved, then $\Gamma$ is identified. To see this, let $F_f \in \mathbb{F}^{2}_b, i \in \N^*$ and $t \geq 0$, and define two real-valued variables
\begin{align*}
\rho^K &  = \int_{\tau^K_{i-1} \wedge t}^{\tau^K_i \wedge t} \int_{\mathcal{M}_F(\X)}F_f(\mu) \Gamma^K(ds \times d \mu)
 \textrm{ and } \rho = \int_{\tau_{i-1} \wedge t}^{\tau_i \wedge t} \int_{\mathcal{M}_F(\X)}F_f(\mu) \Gamma (ds \times d \mu).
\end{align*}
Then certainly $\rho^K \Rightarrow \rho$ as $K \to \infty$. Take any open set $B \in \R$. Then
\begin{align*}
\lim_{K \to \infty }\P \left( \rho^K \in B  \right) & = \lim_{\epsilon \to 0 } \lim_{K \to \infty }\P \left( \rho^K \in B , \ E^K_{i}(\epsilon)\right) \\
& = \lim_{\epsilon \to 0 } \lim_{K \to \infty }\P \left( \int_{\tau^K_{i-1} \wedge t}^{\tau^K_i \wedge t} F_f\left( \widehat{n}_{R^K_i(\epsilon)} \delta_{R^{K,\epsilon}_i} \right) ds \in B , \ E^K_{i}(\epsilon)\right)\\
& = \lim_{\epsilon \to 0 } \lim_{K \to \infty }\P \left(  \left( \tau^K_{i} \wedge t - \tau^K_{i-1} \wedge t\right) F_f\left( \widehat{n}_{R^{K,\epsilon}_i} \delta_{R^{K,\epsilon}_i} \right)  \in B \right)\\
& = \P \left(  \left( \tau_{i} \wedge t - \tau_{i-1} \wedge t\right) F_f\left( \widehat{n}_{R_i} \delta_{R_i} \right)  \in B \right).
\end{align*}
This proves \eqref{occupationmeasurerelation}.

\begin{proof}[Proof of Proposition \ref{propcvfinale}]
For each $i \in \N^*$, we can construct a process $\{ Y^K_i(t) : t \geq 0  \}$ with generator $L^K_0$ such that
\begin{align}
\label{couplingofyks}
Z^K (\tau^K_i +  t) = Y^K_i\left( \frac{t}{K u_K} \right) \textrm{ for all } t \in [0, \tau^K_{i+1} - \tau^K_i).
\end{align}

For $i=1$, we obtain from Assumption \ref{assumption:monomorphicIC} and part (A) of Corollary \ref{maincorollary} that
\begin{align}
\label{fori=1}
\lim_{K \to \infty} \P \left( Y^K_0(t) \in \mathcal{N}_{\epsilon}(x_0) \textrm{ for all } t \in \left[ \frac{\epsilon}{K u_K} , \frac{\epsilon^{-1}}{K u_K} \right]\right) = 1,
\end{align}from which we have \eqref{mainfactfori=0} for $i=1$ thanks to \eqref{couplingofyks} and \eqref{timebetweenmutations}.\\

We now proceed by induction. Let us assume that for $i\geq 1$, we have \eqref{mainfactfori=0}. Let $U^K_i$ denote the type of the new mutant that appears at time $\tau^K_i$. Pick $x,y \in \X$. On the event $\{E^K_i(\epsilon), R^{K,\epsilon}_i = x,U^K_i = y \}$, we have $\textrm{supp}(Z^K \left( \tau^K_i \right)) = \{x,y\}$, $Z^K_{\tau^K_i}\{y\} = 1/K$ and $Z^K_{\tau^K_i}\{x\} \in [ \widehat{n}_x -\epsilon, \widehat{n}_x+\epsilon]$. Using parts (B) and (C) of Corollary \ref{maincorollary} and \eqref{couplingofyks}, we obtain that
\begin{align*}
&\lim_{\epsilon \to 0}\lim_{K \to \infty} \P \left( E^K_{i+1}(\epsilon),R^{K,\epsilon}_{i+1} = x \middle \vert  E^K_i(\epsilon), R^{K,\epsilon}_i = x,U^K_i = y \right) = \left(1- \frac{ [\textrm{Fit}(y,x)]^+}{b(y)} \right)\\
&\textrm{ and }\lim_{\epsilon \to 0}\lim_{K \to \infty} \P \left( E^K_{i+1}(\epsilon),R^{K,\epsilon}_{i+1} = y \middle \vert E^K_i(\epsilon), R^{K,\epsilon}_i= x,U^K_i = y  \right) = \frac{ [\textrm{Fit}(y,x)]^+}{b(y)}.
\end{align*}
Since the distribution of $U^K_i $ conditionally to $\{ E^K_i(\epsilon), R^{K,\epsilon}_i(\tau^K_i) = x\}$ is $m(x,dy)$,
\begin{align}
\label{mainfactfori}
&\lim_{\epsilon \to 0}\lim_{K \to \infty} \P \left( E^K_{i+1}(\epsilon), R^{K,\epsilon}_{i+1} \in  A \middle \vert  E^K_i(\epsilon), R^{K,\epsilon}_i = x \right)\notag \\& = \ind_{A}(x)\int_{\X} \left(1- \frac{ [\textrm{Fit}(y,x)]^+}{b(y)} \right)m(x,dy) + \int_{A} \frac{ [\textrm{Fit}(y,x)]^+}{b(y)} m(x,dy).
\end{align}
This along with \eqref{fori=1} for $i=1$ implies that for any $i \in \N^*$ we have $\lim_{\epsilon \to 0}\lim_{K \to \infty} \P \left( E^K_{i}(\epsilon) \right) = 1.$ This concludes the proof of Proposition \ref{propcvfinale} and the proof of Theorem \ref{theorem:convergenceTSS}.
\end{proof}


\renewcommand {\theequation}{A.\arabic{equation}}
\appendix
\setcounter{equation}{0}

\section{Appendix.}

We start with recalling some useful estimates for birth and death process (Lemma \ref{branchingprocessbehaviourlemma}). Coupling facts between the components of a 2D-birth and death processes on the quadrant (see \eqref{transitions}) and 1D-birth and death processes are then given in Proposition \ref{proposition:couplingwithbranchingprocesses}. Both results are then used to obtain estimates for the 2D-birth and death processes on the quadrant.

\begin{definition}
\label{branchingprocessdefinition}
For any $b, d \in \R_{+}$ and $n \in \N$, let ${\bf P}(b,d,n)$ denote the law of the $\N$-valued continuous time branching process starting at $n$ with birth rate $b$ and death rate $d$. For convenience, we will consider ${\bf P}(b,\infty,n)$ to be the law of the process that is $0$ at all times.
\end{definition}

\begin{lemma}
\label{branchingprocessbehaviourlemma}
For each positive integer $K$ let $\{ B_K(t) : t \geq 0  \}$ be a continuous time branching process with law ${\bf P}(b,d,K)$ with $b \neq d$. Pick an $\epsilon \in [0,1]$ and define a stopping time
\begin{align*}
\sigma_K = \inf \left\{ t \geq 0 :  B^K(t) \leq K (1-\epsilon) \textrm{ or } B^K(t) > K (1+\epsilon)\right\}.
\end{align*}
Then we have the following.
\begin{itemize}
\item[(A)] If $b < d$ then $\P \left( B^K\left(\sigma_K \right) > K (1+\epsilon) \right) \leq \exp\left(-K \epsilon \log (d/b)\right)$\\
and if $b>d$ then $\P \left( B^K\left(\sigma_K \right) \leq K (1-\epsilon) \right) \leq \exp\left(-K \epsilon \log (b/d)\right)$.\\
\item[(B)] If $bd=0$ and $\epsilon \notin \{0,1\}$ then $\inf_{K \geq 1} \E \left( \sigma_K \right) > 0 \textrm{ and } \sup_{K \geq 1} \E \left( \sigma^2_K \right) < \infty.$\\
In parts (C) and (D) let $t_K$ be any $\N^*$-valued sequence satisfying $\log K \ll t_K$
\item[(C)] If $\epsilon = 1$ and if $b < d$ then
$\lim_{K \to \infty} \P \left( \sigma_K \leq t_K, B^K\left(\sigma_K \right) = 0 \right) = 1.$\\
\item[(D)] Let $\{Y(t) : t \geq 0  \}$ be a branching process with law ${\bf P}(b,d,1)$, where $b>d$. For any $\epsilon>0$ define
\begin{align*}
\gamma_K = \inf \left\{ t \geq 0 :  Y(t) = 0 \textrm{ or } Y(t) \geq K \epsilon)\right\}.
\end{align*}
Then
\[1- \lim_{K \to \infty} \P \left( \gamma_K \leq t_K, Y\left(\gamma_K \right) \geq K \epsilon \right) = \lim_{K \to \infty} \P \left( \gamma_K \leq t_K,Y\left(\gamma_K \right) = 0 \right) = \frac{d}{b}.\]
\end{itemize}
\end{lemma}

\begin{proof}
Note that $\sigma_K < \infty$ a.s. since a branching process either goes to $0$ or to $\infty$ almost surely. We can easily check that $M_K(t) = ( d/b)^{B^K(t \wedge \sigma_K)}$
is a bounded martingale. Then by the optional sampling theorem we get
\begin{align}
\label{optionalsamplingequation}
\E \left( M_K(\sigma_K)\right) = \E \left(M_K(0) \right) = \left( \frac{d}{b}\right)^{K}.
\end{align}
If $b<d$ then $ \E \left( M_K(\sigma_K)\right)  \geq \P \left( B^K\left(\sigma_K \right) > K (1+\epsilon) \right) (d/n)^{ K (1+\epsilon)}$ and if $b>d$ then
$\E \left( M_K(\sigma_K)\right)  \geq \P \left( B^K\left(\sigma_K \right) \leq K (1-\epsilon) \right) ( d/b)^{K(1-\epsilon)}$.
Substituting these estimates in \eqref{optionalsamplingequation} proves part (A). \\

For part (B) we assume that $d=0$ and $b>0$. The case where $b=0$ and $d>0$ is similar. We can also assume that $K(1+\epsilon)$ is a positive integer. Since $d=0$, the process $B^K$ is monotonically increasing and hence $ B^K(\sigma_K) = K(1+\epsilon)$. We can show with Dynkin's formula that
$ B^K(t) - b \int_{0}^{t} B^K(s)ds$ is a martingale, and obtain by the optional sampling theorem that
\begin{align*}
K = \E \left( B^K(\sigma_K)\right) -b \E \left( \int_{0}^{\sigma_K} B^K(s) ds \right) = K (1+\epsilon) - b \E \left( \int_{0}^{\sigma_K} B^K(s) ds\right).
\end{align*}
Since $B^K(t) \in [K,K(1+\epsilon)]$ for $t \leq \sigma_K$ we have
$K \E \left( \sigma_K\right) \leq \E \left( \int_{0}^{\sigma_K} B^K(s)ds\right) \leq  K(1+\epsilon) \E \left( \sigma_K\right),$
and hence
\begin{align}
\label{firstmomentofsigmak}
\frac{ \epsilon}{b(1+\epsilon)} \leq \inf_{K \geq 1} \E \left( \sigma_K\right) \leq \sup_{K \geq 1} \E \left( \sigma_K\right) \leq \frac{ \epsilon}{b}.
\end{align}
From the fact that $ t B^K(t) - \int_{0}^{t} B^K(s)(bs+1)ds$
is also a martingale, we obtain that
\[  K \E \left( \frac{b}{2} \sigma_K^2 + \sigma_K\right)\leq \E \left( \int_{0}^{\sigma_K} B^K(s)(bs+1)ds \right) = \E \left( \sigma_K B^K(\sigma_K) \right) = K(1+\epsilon) \E \left(\sigma_K \right).\]
This along with \eqref{firstmomentofsigmak}, proves part (B). Parts (C) and (D) follow directly from \cite[Th. 4]{champagnat3}.
\end{proof}

\begin{definition}
\label{infandsup}
Suppose that $d$ is a function from $\N \times \N$ to $\R_{+}$. Then for any set $\mathcal{S} \subset \R_{+} \times \R_{+} $ we define
\begin{align*}
\inf d(\mathcal{S}) &= \inf \left\{ d(m,n) : (m,n) \in \mathcal{S} \cap (\N \times \N) \right\} \\
\textrm{  and  } \sup d(\mathcal{S}) &= \sup \left\{ d(m,n) : (m,n) \in \mathcal{S} \cap (\N \times \N) \right\}.
\end{align*}
\end{definition}

\begin{proposition}
\label{proposition:couplingwithbranchingprocesses}
Let $\{N(t) = \left( N_1(t), N_2(t) \right)\}$ be a $\N \times \N$-valued pure jump Markov process with the following transition rates.
\[\begin{array}{cc}
m b_1  & \textrm{ from } (m,n) \textrm{ to }  (m+1,n) \\
n b_2  & \textrm{ from } (m,n) \textrm{ to }  (m,n+1) \\
m d_1(m,n) &  \textrm{ from } (m,n) \textrm{ to }  (m-1,n) \\
n d_2(m,n) &  \textrm{ from } (m,n) \textrm{ to }  (m,n-1) .
\end{array}
\]
Here $b_1,b_2$ are positive constants and $d_1,d_2$ are functions from $\N \times \N$ to $\R_{+}$. Suppose that there is a set $\mathcal{S} \subset \R_{+} \times \R_{+}$ and constants $d^{+}_1,d^{-}_1, d^{+}_2, d^{-}_2 \in [0,\infty]$ such that
\begin{align*}
d^{-}_1 &\leq \inf d_1(\mathcal{S})\leq \sup d_1(\mathcal{S}) \leq d^{+}_1  \textrm{ and } \\
d^{-}_2 &\leq \inf d_2(\mathcal{S})\leq \sup d_2(\mathcal{S})\leq d^{+}_2
\end{align*}
Assume that $(N_1(0),N_2(0)) \in \mathcal{S}$ and let $T_\mathcal{S}$ be the random time defined by
\begin{align*}
T_\mathcal{S} = \inf\{ t \geq 0 : N(t) \notin \mathcal{S} \}.
\end{align*}
Let $z^{+}_1,z^{-}_1,z^{+}_2,z^{-}_1$ be positive integers satisfying $z^{-}_1 \leq N_1(0)\leq z^{+}_1$ and $z^{-}_2 \leq N_2(0)\leq z^{+}_2$.
Then on the same probability space as $N$, we can construct four $\N$-valued processes $B^{+}_1,B^{-}_1$, $B^{+}_2$ and $B^{-}_2$ with laws ${\bf P}(b_1,d^{-}_1,z^{+}_1)$, ${\bf P}(b_1,d^{+}_1,z^{-}_1)$, ${\bf P}(b_2,d^{-}_2,z^{+}_2)$, ${\bf P}(b_2,d^{+}_2,z^{-}_1)$ such for all $t \leq T_\mathcal{S}$ the following relations are satisfied almost surely,
\begin{align*}
 B^{-}_1(t) \leq N_1(t) \leq B^{+}_1(t) \quad\textrm{ and } \quad
 B^{-}_2(t) \leq N_2(t) \leq B^{+}_2(t) .
\end{align*}
\end{proposition}
\begin{proof}The proof follows from direct coupling of these processes.\end{proof}

\begin{proof}[Proof of Proposition \ref{proposition:couplingwithbranchingprocesses2}]
We prove part (A). Without loss of generality, we can assume that $K \eta, K\eta(1-\epsilon),K\eta(1+\epsilon)$ are positive integers and $N^K_1(0) = K \eta$. Define
\begin{align*}
\sigma_K = \inf \left\{ t \geq 0 :  N^K(t) \notin [K\eta(1-\epsilon),K\eta(1+\epsilon)] \times [C_K,D_K) \right\}.
\end{align*}
In order for the process $N^K_1$, started at $K\eta$, to go beyond level $B_K$ it must exit the interval $[K\eta(1-\epsilon),K\eta(1+\epsilon)]$ from above. The probability to go from $K\eta$ to $K\eta(1+\epsilon)$ is exponentially small. Indeed, since $b < d^{-}_1(\eta,\epsilon)$, by Proposition \ref{proposition:couplingwithbranchingprocesses} we can construct a coupled subcritical branching process $B^K_{+}$ with law ${\bf P}(b_1,d^{-}_1(\eta,\epsilon),K)$ such that $B^K_{+}(t) \geq N^K_1(t)$ for all $t \leq \sigma_K$ almost surely. Hence if $\gamma_K$ is the first time the process $B^K_{+}$ leaves the set $[K\eta(1-\epsilon),K\eta(1+\epsilon)]$ then from part (A) of Lemma \ref{branchingprocessbehaviourlemma} we obtain that for some $c>0$
\begin{align}
\label{subcriticalexitprobability}
\P \left( N^K_1(\sigma_K) > K\eta(1+\epsilon)  \right) \leq \P \left( B^K_{+}(\gamma_K) > K\eta(1+\epsilon)\right) \leq e^{-c K}.
\end{align}
Before $N^K_1$ exists $[A_K,B_K)$ from above, this process crosses the interval $[K\eta(1-\epsilon),K\eta]$ several times. Let $\rho_K$ be the number (possibly 0) of these passages in the interval $[0,\mathcal{T}_{\mathcal{S}_K}]$. Because of \eqref{subcriticalexitprobability}, for any $n\in \N$,
\begin{align}
\label{exitprobabilityestimate}
\P \left( \rho_K < n  \right) \leq n e^{-cK}.
\end{align}
Let $n_K =[ 1/Ku_K ]^2$. Then,
\begin{multline}
 \P \left( T_{\mathcal{S}_K} \leq \frac{T}{K u_K} , N^K_1 \left( T_{\mathcal{S}_K} \right) \geq B_K \right) \\
\begin{aligned}
& \leq \P \left( T_{\mathcal{S}_K} \leq \frac{T}{K u_K} , N^K_1 \left( T_{\mathcal{S}_K} \right) \geq K\eta(1+\epsilon) , \rho_K \geq n_K \right) + \P (\rho_K < n_K)\\
&\leq \P \left( \sum_{i=1}^{n_K}  \tau_{K,i} < \frac{T}{K u_K}\right) + \P (\rho_K < n_K),\label{etape2}
\end{aligned}
\end{multline}where the $\tau_{K,i}$ denotes the durations of the $n_K$'s first passages of $N^K_1$ from $K\eta$ to $K\eta(1+\epsilon)$. The last term in the r.h.s. of \eqref{etape2} tends to 0 with $K$ due to our choice of $n_K$, \eqref{exitprobabilityestimate} and \eqref{timescaleseparation}. Using the couplings of Proposition \ref{proposition:couplingwithbranchingprocesses}, it is possible to dominate $N^K_1$ by a pure birth process with distribution ${\bf P}(b_1,0,K\eta(1-\epsilon))$ on each of its excursions of $N^K_1$ from $K\eta(1-\epsilon)$ to $K\eta$. Thus, each $\tau_{K,i}$ can be bounded above by the time $\bar{\sigma}_{K,i}$ needed by a pure birth process with distribution ${\bf P}(b_1,0,K\eta(1-\epsilon))$ to reach $K\eta$. The $\bar{\sigma}_{K,i}$'s can be chosen to be i.i.d. and by Lemma \ref{branchingprocessbehaviourlemma}, Part (B), $
\inf_{K \geq 1} \E \left( \tau_{K,1} \right) > 0$ and $\sup_{K \geq 1} \E \big( \tau^2_{K,1} \big) < \infty$.
Applying Chebychev's inequality we get
\begin{align*}
\P \left( \sum_{i=1}^{n_K} \tau_{K,i} \leq \frac{T}{K u_K} \right) & \leq  \P \left( \frac{1}{n_K} \sum_{i=1}^{n_K} \bar{\sigma}_{K,i} - \E \left( \bar{\sigma}_{K,i} \right)  \leq \frac{T}{K u_K n_K} -\E \left( \bar{\sigma}_{K,1} \right) \right) \\
& \leq \frac{ \E \big( \bar{\sigma}_{K,1}^2 \big) }{n_K \left( \E \left( \bar{\sigma}_{K,1} \right) - \frac{T}{K u_K n_K} \right)^{+}}.
\end{align*}
As $K \to \infty$, $T/(K u_K n_K) \to 0$ and $n_K \to \infty$. This proves part (A). Proof of part (B) is similar and part (C) is a direct consequence of parts (A) and (B).\\

We now prove part (D). Using Proposition \ref{proposition:couplingwithbranchingprocesses} and since $b_2<d^{-}_2(\epsilon)$ and $N^K_2(0) \leq K \epsilon /2$, $N^K_2$ can be dominated on $[0,T_{\mathcal{S}_K}]$ by a subcritical branching process $B^K_{+}$ with distribution ${\bf P}(b_2,d^{-}_2(\epsilon), [K \epsilon /2])$. Let
\begin{align}
\label{defgammak+}
 \gamma_K^{+} = \inf \{ t \geq 0 : B^K_{+}(t) = 0 \textrm{ or } B^K_{+}(t) \geq K \epsilon \}.
\end{align}
The coupling ensures that if $T_{\mathcal{S}_K} \geq \gamma_K^+$ and $B^K_{+}(\gamma_K^+) = 0$ then $N^K_2(\gamma_K^+)  = 0$ and $T_{\mathcal{S}_K} = \gamma_K^+$. Since $B^K_{+}(t)\in (0,K \epsilon)$ for $t \in [0,\gamma_K^+)$ we can also see that if $T_{\mathcal{S}_K} < \gamma_K^+$ then either $N^K_2 \left( T_{\mathcal{S}_K} \right) = 0$ or $N^K_1 \left( T_{\mathcal{S}_K} \right) \notin [A_K,B_K)$. Hence
\begin{align}
&\P \left( \gamma_K^+ \leq t_K, B^K_{+}(\gamma_K^+) = 0 \right)\nonumber\\
& \leq \P \left( \gamma_K^+ \leq t_K , B^K_{+}(\gamma_K^+) = 0 , T_{\mathcal{S}_K} \geq \gamma_K^+\right) + \P \left( \gamma_K^+ \leq t_K ,  T_{\mathcal{S}_K} < \gamma_K^+\right)\nonumber\\
& \leq  \P \left( T_{\mathcal{S}_K} \leq t_K, N^K_2 \left( T_{\mathcal{S}_K} \right) = 0 , T_{\mathcal{S}_K} = \gamma_K^+\right)
+ \P \left( T_{\mathcal{S}_K} \leq t_K, N^K_2 \left( T_{\mathcal{S}_K} \right) = 0 , T_{\mathcal{S}_K} < \gamma_K^+\right) \nonumber\\
& + \P \left( T_{\mathcal{S}_K} \leq t_K, N^K_1 \left( T_{\mathcal{S}_K} \right) \notin [A_K,B_K), T_{\mathcal{S}_K} < \gamma_K^+ \right)\nonumber\\
& \leq \P \left( T_{\mathcal{S}_K} \leq t_K, N^K_2 \left( T_{\mathcal{S}_K} \right) = 0 \right) +
\P \left( T_{\mathcal{S}_K} \leq t_K, N^K_1 \left( T_{\mathcal{S}_K} \right) \notin [A_K,B_K) \right).\label{etape3}
\end{align}
Due to part (C), the limit of the second term in the r.h.s. of \eqref{etape3} is $0$. Since part (C) of Lemma \ref{branchingprocessbehaviourlemma} tells us that the limit of the l.h.s. is 1, this proves part (D).\\

We now prove part (E). Since $b_2 > d^{+}_2(\epsilon) > d^{-}_2(\epsilon)$ and $N^K_2(0)=1$, using Proposition \ref{proposition:couplingwithbranchingprocesses}, we can construct two coupled supercritical branching processes $B^K_{-}, B^K_{+}$ with distributions ${\bf P}(b_2, d^{+}_2(\epsilon),1), {\bf P}(b_2, d^{-}_2(\epsilon),1)$ such that $B^K_{-}(t) \leq N^K_2(t) \leq B^K_{+}(t)$ for all $t \leq T_{\mathcal{S}_K}$ almost surely. Let $\gamma^+_K$ be given by (\ref{defgammak+}) and $\gamma^{-}_K$ be the corresponding stopping time for $B^K_{-}$. The fact that $B^K_{-}$ is below $N^K_2$ until time $T_{\mathcal{S}_K}$, \eqref{etape3} and part (C) allow us to prove that:
\begin{multline}\lim_{K\rightarrow +\infty}\P \left( \gamma_K^+ \leq t_K, B^K_{+}(\gamma_K^+) = 0 \right)\leq \lim_{K\rightarrow +\infty} \P \left( T_{\mathcal{S}_K} \leq t_K, N^K_2 \left( T_{\mathcal{S}_K} \right) = 0 \right)\\
\leq \lim_{K\rightarrow +\infty}\P \left( \gamma_K^- \leq t_K, B^K_{-}(\gamma_K^-) = 0 \right).\label{dominationabovebelowlimit}
\end{multline}
Since $B^K_{+}$ is above $N^K_2$ until time $T_{\mathcal{S}_K}$, we obtain:
\begin{equation}
\lim_{K\rightarrow +\infty} \P\left(T_{\mathcal{S}_K}\leq t_K,\ N^K_2(T_{\mathcal{S}_K})\geq K\epsilon\right)\leq \lim_{K\rightarrow+\infty} \P\left(\gamma_K^+ \leq t_K,\ B^K_{+}(\gamma_K^+)\geq K\epsilon\right).
\label{upperdomination1}
\end{equation}
Finally, a proof similar to the one of \eqref{etape3}, we can prove that:
\begin{align}
\label{upperdomination2}
\lim_{K \to \infty}\P \left( \gamma_K^- \leq t_K, B^K_{-}(\gamma_K^-) \geq  K \epsilon \right) \leq \lim_{K \to \infty} \P \left( T_{\mathcal{S}_K} \leq t_K, N^K_2 \left( T_{\mathcal{S}_K} \right) \geq K \epsilon \right).
\end{align}
From part (D) of Lemma \ref{branchingprocessbehaviourlemma} and Assumption \ref{timescaleseparation} we can deduce that
\begin{align*}
&1 - \lim_{K \to \infty} \P \left( \gamma^{+}_K  \leq t_K,  B^K_{+} \left( \gamma^{+}_K  \right) \geq K \epsilon \right) = \lim_{K \to \infty} \P \left( \gamma^{+}_K  \leq t_K,  B^K_{+} \left( \gamma^{+}_K  \right) = 0 \right) = \frac{d^{-}_2(\epsilon)}{b_2}\\
\mbox{ and }\quad  &  1-\lim_{K \to \infty} \P \left( \gamma^{-}_K  \leq t_K,  B^K_{-} \left( \gamma^{-}_K  \right) \geq K \epsilon \right) =  \lim_{K \to \infty} \P \left( \gamma^{-}_K  \leq t_K,  B^K_{-} \left( \gamma^{-}_K  \right) = 0 \right) = \frac{d^{+}_2(\epsilon)}{b_2}.
\end{align*}
These relations along with \eqref{dominationabovebelowlimit}, \eqref{upperdomination1} and \eqref{upperdomination2} prove part (E).
\end{proof}

\bigskip

\noindent \textbf{Acknowledgments:} The authors thank Sylvie M\'el\'eard for invaluable discussion. This work benefited from the support of the ANR MANEGE (ANR-09-BLAN-0215), from the Chair ``Mod\'elisation Math\'ematique et Biodiversit\'e" of Veolia Environnement-Ecole Polytechnique-Museum National d'Histoire Naturelle-Fondation X.

{\footnotesize
\providecommand{\noopsort}[1]{}\providecommand{\noopsort}[1]{}\providecommand{\noopsort}[1]{}\providecommand{\noopsort}[1]{}

}

\begin{thebibliography}{10}

\bibitem{billingsley}
P.~Billingsley.
\newblock {\em Convergence of Probability Measures}.
\newblock John Wiley \& Sons, New York, 1968.

\bibitem{champagnat3}
N.~Champagnat.
\newblock A microscopic interpretation for adaptative dynamics trait
  substitution sequence models.
\newblock {\em Stochastic Processes and their Applications}, 116(8): 1127--1160, 2006.

\bibitem{champagnatferrieremeleard2}
N.~Champagnat, R.~Ferri\`{e}re, and S.~M\'{e}l\'{e}ard.
\newblock Individual-based probabilistic models of adpatative evolution and
  various scaling approximations.
\newblock In R.C. Dalang, M. Dozzi and F. Russo, editors, {\em Proceedings of the 5th seminar on Stochastic Analysis, Random
  Fields and Applications, Ascona, May 2005}, Volume 59 of Progress in Probability, pages 75-114, Birkhauser, Basel, 2006.

\bibitem{champagnatferrieremeleard}
N.~Champagnat, R.~Ferri\`{e}re, and S.~M\'{e}l\'{e}ard.
\newblock Unifying evolutionary dynamics: from individual stochastic processes
  to macroscopic models via timescale separation.
\newblock {\em Theoretical Population Biology}, 69: 297--321, 2006.

\bibitem{champagnatmeleard2011}
N.~Champagnat and S.~M\'{e}l\'{e}ard.
\newblock Polymorphic evolution sequence and evolutionary branching.
\newblock {\em Probability Theory and Related Fields} 151: 45--94, 2011

\bibitem{colletmeleardmetz}
P.~Collet, S.~M\'el\'eard, and J.A.J. Metz.
\newblock A rigorous model study of the adaptative dynamics of Mendelian
  diploids.
\newblock 2011.
\newblock http://arxiv.org/abs/1111.6234.

\bibitem{dawson}
D.~A. Dawson.
\newblock Mesure-valued markov processes.
\newblock In P.-L. Hennequin, editor, {\em Ecole d'Et\'{e} de probabilit\'{e}s de
  Saint-Flour XXI}, volume 1541 of {\em Lectures Notes in Math.}, pages 1--260, Springer,
  New York, 1993.

\bibitem{DieckmannHeinoParvinen} U. Dieckmann, M. Heino and K. Parvinen
\newblock The adaptive dynamics of function-valued traits.
\newblock {\em Journal of theoretical Biology} 241(2): 370--389, 2006.

\bibitem{DieckmannLaw1996}
U.~Dieckmann and R.~Law.
\newblock The dynamical theory of coevolution: A derivation from stochastic ecological processes.
\newblock {\em Journal of Mathematical Biology}, 34(5--6): 579--612, 1966.

\bibitem{durinxmetzmeszena}
M.~Durinx, J.A.J. Metz, and G.~Mesz\'{e}na.
\newblock Adaptive dynamics for physiologically structured models.
\newblock {\em Journal of Mathematical Biology}, 56(5): 673--742, 2008.

\bibitem{ethierkurtz}
S.N. Ethier and T.G. Kurtz.
\newblock {\em Markov Processus, Characterization and Convergence}.
\newblock John Wiley \& Sons, New York, 1986.

\bibitem{geritzkisdimescenametz}
S.A.H. Geritz, \'E. Kisdi, G. Mesz\'ena and J.A.J. Metz
\newblock Evolutionarily singular strategies and the adaptive growth and branching of the evolutionary tree.
\newblock {\em Evolutionary Ecology}, 12: 35--57.

\bibitem{guptametztran}
A.~Gupta, J.A.J. Metz, and V.C. Tran.
\newblock work in progress.

\bibitem{kallenberg}
O.~Kallenberg.
\newblock {\em Random Measures}.
\newblock Academic Press, 1983.

\bibitem{kurtz}
T.G. Kurtz.
\newblock {\em Approximation of population processes.}
\newblock  SIAM, Philadelphia, PA, 1981.

\bibitem{kurtzaveraging}
T.G. Kurtz.
\newblock Averaging for martingale problems and stochastic approximation.
\newblock In I. Karatzas and D. Ocone, editors, {\em Applied stochastic analysis (New Brunswick, NJ, 1991)}, volume 177 of {\em Lectures Notes in Control and Inform. Sci.},
  pages 186--209, Springer, Berlin, 1992.

\bibitem{meleardmetztran}
S.~M\'{e}l\'{e}ard, J.A.J. Metz, and V.C. Tran.
\newblock Limiting {F}eller diffusions for logistic populations with
  age-structure.
\newblock In {\em Proceedings of the 58th World Statistics Congress - ISI2011}.
  ISI, 2011.
\newblock Prepring HAL 00595928.

\bibitem{meleardtran}
S.~M\'{e}l\'{e}ard and V.C. Tran.
\newblock Trait substitution sequence process and canonical equation for
  age-structured populations.
\newblock {\em Journal of Mathematical Biology}, 58(6): 881--921, 2009.

\bibitem{meleardtransuperage}
S.~M\'{e}l\'{e}ard and V.C. Tran.
\newblock Slow and fast scales for superprocess limits of age-structured
  populations.
\newblock {\em Stochastic Processes and their Applications}, 122(1): 250--276,
  January 2012.

\bibitem{metznisbetgeritz}
J.A.J. Metz, R.M.Nisbet and S.A.H. Geritz
\newblock How should we define fitness for general ecological scenarios.
\newblock {\em Trens in Ecology and Evolution}, 7(6): 198--202, 1992.

\bibitem{metzgeritzmeszenajacobsheerwaarden}
J.A.J. Metz, S.A.H. Geritz, G.~Mesz\'{e}na, F.A.J. Jacobs, and J.S.~Van
  Heerwaarden.
\newblock Adaptative dynamics, a geometrical study of the consequences of
  nearly faithful reproduction.
\newblock In S.J. Van Strien and S.M. Verduyn Lunel, editors, {\em  Stochastic and
  Spatial Structures of Dynamical Systems}, Volume 45 of {\em KNAW Verhandelingen Afd. Natuurkunde, Eerste reeks}, pages 183--231, North Holland, Amsterdam,1996.

\bibitem{ParvinenDieckmannHeino} K. Parvinen, U. Dieckmann and M. Heino
\newblock  Function-valued adaptive dynamics and the calculus of variations.
\newblock {\em Journal of Mathematical Biology}, 52(1): 1--26, 2006.


\end{thebibliography}


\end{document}